\documentclass{article}
\usepackage[english]{babel}
\usepackage{amsmath,amssymb,amsfonts,amsthm,ulem}
\usepackage[inner=3.4cm,outer=3.65cm,bottom=3cm]{geometry}

\usepackage{hyperref}
\usepackage{color}
\usepackage{verbatim}

\theoremstyle{remark}
\newtheorem{remark}{Remark}[section]
\theoremstyle{definition}
\newtheorem{theorem}{Theorem}[section]
\newtheorem{definition}[theorem]{Definition}

\newtheorem{proposition}[theorem]{Proposition}
\newtheorem{lemma}[theorem]{Lemma}
\newtheorem{corollary}[theorem]{Corollary}
\newtheorem{hypothesis}[theorem]{Hypothesis}

\DeclareMathOperator{\deriv}{d}
\newcommand{\gammaciapo}{\widehat{\gamma}}

\DeclareMathOperator{\TV}{TV}
\newcommand{\ddt}{\frac{\deriv\!{}}{\dit}}
\newcommand{\bn}{\boldsymbol{n}}
\newcommand{\dit}{\deriv\!t}
\newcommand{\duav}[1]{\langle{#1}\rangle}

\newcommand{\dn}{\partial_{\bn}}
\DeclareMathOperator{\Bv}{BV}
\DeclareMathOperator{\dive}{div}

\DeclareMathOperator{\R}{\mathbb{R}}

\DeclareMathOperator{\C}{\mathcal{C}}

\DeclareMathOperator{\N}{\mathbb{N}}

\newcommand{\ov}[1]{\overline{#1}}
\DeclareMathOperator{\sign}{sign}
\DeclareMathOperator{\Id}{Id}
\newcommand{\ra}{\rightarrow}
\newcommand{\de}{\mathrm{d}}

\newcommand{\f}[1]{{\pmb{ #1}}}
\DeclareMathOperator{\di}{\nabla\cdot}

\newcommand{\ts}{\tilde{\sigma}}
\newcommand{\tp}{\tilde{\varphi}}
\newcommand{\tmu}{\tilde{\mu}}
\newcommand{\rs}{{\sigma_{\varepsilon}}}
\newcommand{\rp}{{\varphi_{\varepsilon}}}
\newcommand{\rmu}{{\mu_{\varepsilon}}}

\renewcommand{\t}{\partial_t}

\newcommand{\CA}{C_\alpha}

\newcommand{\calE}{\mathcal{E}}
\newcommand{\calF}{\mathcal{F}}
\newcommand{\calD}{\mathcal{D}}
\newcommand{\calR}{\mathcal{R}}

\newcommand{\RR}{\mathbb{R}}

\newcommand{\NN}{\mathbb{N}}

\newcommand{\fhi}{\varphi}

\newcommand{\epsi}{\varepsilon}
\newcommand{\ee}{_\varepsilon}

\newcommand{\io}{\int_\Omega}

\date{Version of the draft: \today}
\begin{document}

\begin{center}
		
		{\LARGE \textbf{Weak solutions and weak-strong uniqueness for \\
        a Cahn-Hilliard type model with chemotaxis}}

		\vskip0.5cm
		
		{\large\textsc{Robert Lasarzik$^1$}} \\
		{\normalsize e-mail: \texttt{robert.lasarzik@wias-berlin.de}} \\
		\vskip0.35cm
		
		{\large\textsc{Elisabetta Rocca$^2$}} \\
		{\normalsize e-mail: \texttt{elisabetta.rocca@unipv.it}} \\
		\vskip0.35cm

        {\large\textsc{Giulio Schimperna$^2$}} \\
		{\normalsize e-mail: \texttt{giulio.schimperna@unipv.it}} \\
		\vskip0.35cm
		
		{\footnotesize $^1$Weierstraß-Institut,
Mohrenstraße 39,
10117 Berlin}
		\vskip0.1cm

        {\footnotesize$^2$Department of Mathematics ``F. Casorati'', University of Pavia \& \\IMATI-C.N.R., 
          Via Ferrata 5, 27100 Pavia, Italy}
        \vskip0.5cm
		
	\end{center}
	

\begin{abstract}
\noindent
We prove existence of weak solutions and weak-strong uniqueness for a mathematical model which couples the evolution of a phase-parameter $\fhi$ satisfying a Cahn-Hilliard type relation  with the one of an additional variable $\sigma$ influencing the phase separation process. The main application of the model refers to cancer growth processes, where $\sigma$ may represent the concentration of a chemical substance affecting the evolution of the tumor, and is governed by a nonlinear parabolic equation characterized by a cross-diffusion term alike that occurring in the Keller-Segel model for chemotaxis. This term is also responsible for the most relevant difficulties in the mathematical analysis of the system. 
Complementing previous results on the model, we  prove here global in time existence for a very weak notion of solution to which a suitable energy imbalance and a 
logarithmic inequality for the nutrient are added. Noting that the system also admits local 
in time ``strong'' solutions, we can also exhibit a weak-strong uniqueness result whose 
proof exploits in an essential way the entropy-type inequality satisfied by weak solutions.
	\vskip3mm

		\noindent {\bf Key words:} Cahn-Hilliard equation, chemotaxis, Keller-Segel model, weak solutions, 
        weak-strong uniqueness, tumor growth models.

		\vskip3mm 
		
		\noindent {\bf AMS (MOS) Subject Classification:} 
        35K61, 
        35D30, 
        35Q92, 
        92C50. 
\end{abstract}
\section{Introduction}

In this article, we consider the PDE system
\begin{subequations}\label{eq}
\begin{align}
  \t \varphi - \Delta \mu = 0 ,\qquad \mu = - \Delta \varphi + F'(\varphi) - \chi \sigma, 
   &\quad \text{in }\,\Omega \times (0,T),\label{eq:1}\\
  \t \sigma - \dive (\sigma \nabla (\ln \sigma + \chi ( 1-\varphi )) ) = \alpha(\varphi, \sigma) \sigma,
    & \quad \text{in }\Omega \times (0,T),\label{eq:2}
\end{align}
equipped with the boundary conditions 
\begin{align}\label{bound:1c} 
\f n \cdot \nabla \mu = 0 = \f n \cdot \nabla \varphi , \qquad
  \f n \cdot ( \sigma \nabla (\ln \sigma + \chi ( 1-\varphi ))) = 0, \quad \text{on }\,\partial\Omega \times (0,T)
\end{align}
and the initial conditions
\begin{align}\label{init:1d}
\varphi (0) = \varphi_0 ,\qquad\sigma(0) =\sigma_0, \quad \text{in } \Omega.
\end{align}
\end{subequations}
Here $\Omega$ is a smooth bounded domain of $\RR^d$, $d\in\{2,3\}$ (we could 
actually consider any dimension $d\geq 2$, however we prefer to state the results
in the physically meaningful cases $d=2,3$ for simplicity), $\f n$ stands for the outer 
normal unit vector to $\partial \Omega$, and $T>0$ is an arbitrarily large, but otherwise fixed, 
final time. In Problem~\eqref{eq}, $\fhi$ represents the order parameter
of a phase separation process and is subject to the Cahn-Hilliard dynamics prescribed by 
system~\eqref{eq:1}, where $\mu$ acts as an auxiliary variable representing the chemical
potential. The additional variable $\sigma$ denotes the concentration of another 
chemical substance affecting the phase separation process. 
The evolution of $\sigma$ is governed by the nonlinear
parabolic equation \eqref{eq:2}, which is characterized by a cross-diffusion term reminiscent
of that occurring in the Keller-Segel model for chemotaxis~\cite{KS}.

The function $F: \R \ra [0,\infty]$ appearing in~\eqref{eq:1}
represents a non-smooth configuration potential, assumed to be $\lambda$-convex (i.e., convex 
up to a quadratic perturbation). Even though more general
choices could be considered, for simplicity we shall directly assume $F$ to be given by a variant of
the Flory-Huggins configuration potential for phase separation, namely 
\begin{equation}\label{Flog}
F(r) = (1+r)\ln (1+r) + (1-r)\ln (1-r) - \frac{\lambda}{2} r^2 ,
\end{equation}
where $\lambda\ge 0$. In particular, the above choice enforces the variable $ \varphi $ 
to take values in the interval $ [-1,1]$ during the evolution, with $\fhi=\pm1$ corresponding
to pure configurations or phases.

The PDE system~\eqref{eq} can be derived  from 
the following free energy functional:
\begin{align}\label{freeenergy}
  {\cal E} (\fhi, \sigma)
    = \frac12 \int_\Omega |\nabla \fhi |^2 \, \de x
    + \io F(\fhi) \, \de x
      + \io \big( \sigma (\ln \sigma - 1) 
      + \chi  \sigma {(1-\fhi)} \big)\, \de x.
\end{align}
In particular, equation \eqref{eq:1} may be obtained as a balance
law by setting
\begin{align*}
  \fhi_t + \dive {\bf J}_\fhi = 0,
\end{align*}
where, as is typical for the Cahn--Hilliard equation, the flux ${\bf J}_\fhi$ 
is prescribed as ${\bf J}_\fhi= -  \nabla \mu$ and  the chemical potential $\mu$ is 
defined as the variational derivative of the free energy with respect to the order parameter,
namely $\mu:= \delta {\cal E}/\delta \fhi$.
Note that also equation \eqref{eq:2} can be obtained as a 
balance law for the nutrient flux ${\bf J}_\sigma $, i.e.,
\begin{align*}
   \sigma_t + \dive {\bf J}_\sigma = \alpha(\fhi,\sigma)\sigma, \quad 
     \text{with }\,{\bf J}_\sigma := - \sigma \nabla \mu_\sigma, \quad 
      \mu_\sigma  :=  \frac {\delta {\cal E}}{\delta \sigma}
            = \ln \sigma + \chi (1-\fhi),
\end{align*}
where the mobility function is proportional to $\sigma $, hence,
in particular, it degenerates (in fact linearly) as $\sigma\searrow 0$.
It is worth observing that this condition also guarantees the minimum principle for
$\sigma$, namely $\sigma$ keeps being nonnegative if that occurs
at the initial time.

Here, $\chi$ is the chemotactic coefficient,   which will be
assumed constant, and strictly positive in view of physical considerations 
(cf.\ \eqref{nutr:V} below), even though, mathematically, our
results would extend to the case $\chi<0$ as well. 
Moreover,  $\alpha$ is a sufficiently regular
function of the two variables $\fhi$ and $\sigma$.
A typical example for $\alpha$ is 
\begin{equation}\label{ex:alpha}
  \alpha(\fhi, \sigma) = h(\fhi)(1- \ell \sigma^p), \quad p\in(0,1],~~\ell>0, 
\end{equation}
where $h$ is a nonnegative and non-decreasing 
interpolation function such that $h(-1)=0$ and $h(1)=1$.
In particular, the above choice, in the case $p=1$, corresponds 
to a logistic behavior of the right-hand side of \eqref{eq:2}
with respect to $\sigma$, as was assumed, for example, in \cite{RSS}.
Regarding the potential \eqref{Flog}, we point out that, even if this choice is generally
considered  the physically most realistic one in the framework of Cahn-Hilliard type models,
it is also the source of a number of mathematical difficulties (cf., e.g., \cite{DD, KNP, MAIMS})
in view of its ``singular'' character. As first noted in \cite{RSS}, however, as one considers
the coupling with the Keller-Segel-like relation \eqref{eq:2}, the choice of a singular potential
is almost compulsory, as it automatically implies the boundedness of the phase parameter in between
$-1$ and $1$. In turn, such a property is crucial in order to guarantee the coercivity 
of the functional \eqref{freeenergy}, and in particular to control the coupling term in it. 

Let us note that the system~\eqref{eq} can be seen as a particular example in a wide class of 
tumor growth models accounting for chemotaxis effects (cf.~\cite{GHW,RSS,GiulioNew}
for further extensions to models also including the velocity field). In this setting,
the phase parameter $\fhi$ represents  the difference between the tumor cells and healthy 
cells volume fractions, and is normalized in such a way that the level sets $\{\fhi=1\}$ 
and $\{\fhi=-1\}$  describe the regions occupied by the pure (“tumor” and “healthy”) phases, 
respectively, whereas the variable $\sigma$ denotes the concentration of the nutrient (glucose 
or oxygen, for instance) responsible for tumor growth.

To our knowledge, the coupling between the Cahn-Hilliard system and the Keller-Segel-like
relation \eqref{eq:2} was first   suggested in \cite{agg}
(as one specific possibility among a wide class of models 
introduced there), and then proposed again, with a more specific
motivation, in \cite{RSS}. In the latter  paper,
the mass balance equation $\eqref{eq:1}_1$ also included a non-zero right-hand 
side $S(\fhi, \sigma)$ representing the mass source and describing how the nutrient supply affects 
the evolution of the tumor. In \cite{RSS},
existence of weak solutions was proved under suitable assumptions on $S$, but uniqueness 
was left as an open problem, at least for general $\alpha$
and in the class of weak solutions (while it was proven to hold for smoother 
solutions). This is why in the present paper we are willing to prove
weak-strong uniqueness of solutions, meaning that, if a weak and a strong solution of the
system stem from the same initial data, then they must coincide at least on the time interval where 
the local strong solution exists. Our results are proved only for zero right-hand side
in $\eqref{eq:1}_1$ (meaning that the tumor mass remains constant in time during
the evolution of the system); however, they may extend to the case of 
an ``affine linear'' source of the form
\begin{equation}\label{RHS}
  S(\fhi,\sigma)=-m\fhi + \ov{h},
\end{equation}
where $\ov{h}$ and $m$ are constants with $m>0$ and $|{\ov h}| < m$. 
On the other hand, as first observed in \cite{CahnHilliardOono}, for 
singular potentials the case of {\it nonlinear}\/ mass source terms 
is much more delicate to deal with from the point of view of uniqueness
(and of weak-strong uniqueness as well), because of the lack of a contractive
estimate for the spatial mean of $\fhi$ (which, instead, it is easy
to prove in the case of \eqref{RHS}).

From the physical point of view, the main justification for considering
our model can be observed by integrating the nutrient equation \eqref{eq:2} 
over an arbitrary   (smooth)  subdomain $V\subset\Omega$. Then, 
using the boundary conditions, one obtains
\begin{equation}\label{nutr:V}
  \ddt \int_V \sigma \, \de x
   = \int_{\partial V} \dn \sigma \, \de S
   + \int_V \alpha(\fhi,\sigma) \sigma \, \de x
   - \chi \int_{\partial V} \sigma \dn \fhi \, \de S;
\end{equation}
namely, as physically expected, the flux of nutrients across $\partial V$,
driven by consumption by tumor cells, is proportional to the actual value of
$\sigma$: the more nutrient is present, the more it may flow away.
 
Actually, it was this simple observation to lead the 
authors of \cite{RSS} to propose the Keller-Segel-like relation \eqref{eq:2}
in place of the more standard expression
\begin{equation}\label{standard-sigma}
  \partial_t\sigma-\Delta(\sigma-\chi\fhi)=\alpha(\fhi,\sigma)\sigma,
\end{equation}
which was considered in most of the previous papers dealing with
diffuse-interface models for tumor growth (cf., e.g., \cite{FLR17, GL17, GL17bis, GT} and 
references therein). We may notice that, if we integrate \eqref{standard-sigma} 
over the reference domain $V$, the factor $\sigma$ in the last integrand in \eqref{nutr:V} disappears,
meaning that the flux of nutrients across $\partial V$ caused by consumption 
by tumor cells is independent of the actual value of $\sigma$. This behavior is
rather questionable from the physical viewpoint and can also result in the loss
of the minimum principle; namely, $\sigma$ satisfying \eqref{standard-sigma}
may well become negative at some point, which is something unexpected
as $\sigma$ represents a concentration. As already observed, such 
a behavior cannot occur in the case of \eqref{eq:2}.

Let us comment a bit more on recent results in the literature related to system 
\eqref{eq}. As noted above, the idea of considering chemotaxis/transport effects 
in equation \eqref{eq:2} comes from the recent paper \cite{RSS}, where various
properties like existence, uniqueness and regularity of solutions,
also depending on the space dimension $d=2,3$, were proved in the case
of a logistic source of the form \eqref{ex:alpha}.
Regarding more specifically the case of dimension $d=2$, still in the case of a 
logistic source in \eqref{eq:2} and also considering a mass source in $\eqref{eq:1}_1$,
more recently some additional results concerning regularity and uniqueness 
have been proved in \cite{GHW}, where transport effects driven 
by a macroscopic velocity field ${\boldsymbol u}$ assumed to satisfy a Brinkman-type law
were added to the model. Velocity effects were also considered in the paper~\cite{GiulioNew}, 
dealing with the case where \eqref{eq:2} accounts for a nonlinear (i.e., power-like) 
{\sl chemotactic sensitivity}\/ $\gamma$; namely, one has
\begin{equation}\label{sigma-power}
  \partial_t\sigma -\Delta\sigma + \chi\dive{(\gamma(\sigma)\nabla\fhi)}
    = \alpha(\fhi,\sigma)\sigma,
\end{equation}
where $\gamma$ is assumed to behave like $\sigma$ for $\sigma\sim 0$ and 
like $\sigma^{2-p}$, $p\in (12/11, 2]$, for large $\sigma$. It is worth observing
that the resulting ``subcritical'' character of the cross-diffusion term 
in \cite{GiulioNew} allows for the choice of a bounded $\alpha$, i.e.\ to
get rid of the stabilizing effect coming from a source of logistic type.
The model considered in \cite{GiulioNew} in the specific case $p=2$,
and the existence result proved there, will also serve as an approximation for our 
system \eqref{eq}. We also quote the recent paper \cite{SchimpSeg}, where the long-time behavior 
of the system considered in \cite{RSS} (including in particular the logistic source 
in \eqref{eq:2}, but neglecting the mass source in \eqref{eq:1}$_1$) was analyzed
from the point of view of infinite-dimensional dynamical systems proving the existence
of the global attractor for the associated evolutionary process. 
Finally, we would like to mention the recent 
contribution \cite{AS}, where the authors introduce a different approach to 
multiphase tumor growth models including chemotaxis in the first Cahn-Hilliard
equation $\eqref{eq:1}_1$, so that   it is the tumor mass that is ``transported''  by the nutrient in their setting. 

Comparing our specific results with the mathematical literature in the field, 
we point out
that, while weak-strong uniqueness has been extensively investigated in the context of
phase-field models (cf., e.g. \cite{ourpreviousone}), also in connection to liquid crystal flows 
(cf., e.g., \cite{EL18}) or damage processes
(cf., e.g., \cite{LRR24}), to our knowledge the present paper is one of the first 
works where a phase field model for tumor growth including a ``singular'' potential
is addressed under this perspective.

In order to briefly introduce our mathematical approach, we point out once
more that the quadratic cross diffusion term occurring in \eqref{eq:2} 
and ``inherited'' from the Keller--Segel (KS) model \cite{KS}
represents the main source of analytical complications. In particular, at least in
three space dimensions, global existence for system \eqref{eq}
{\sl without the logistic degradation term in \eqref{eq:2}}, to our knowledge
was still an open problem. Actually, the results proved in \cite{RSS} strongly
rely on the presence of the logistic source, whereas the ``nonlinear
sensitivity'' considered in \cite{GiulioNew} prescribes the cross-diffusion
term to have a ``subcritical'' character. In the present contribution, we
can give a positive answer to the problem of existence of 
solutions for \eqref{eq} (i.e., considering ``linear'' sensitivity
and omitting logistic sources); our notion of solution, however, will 
be weaker compared to that considered, say, in \cite{RSS,GiulioNew}. 
The key point in our argument consists in considering a sort
of renormalized version of \eqref{eq:2}, which, roughly speaking,
is obtained by dividing \eqref{eq:2} by $\sigma$ and rearranging the 
diffusion terms by means of suitable chain-rule formulas; see, for instance, \eqref{giu:a5} below.
This procedure, in turn, is admissible thanks to a reinforced version of the 
minimum principle for $\sigma$; namely, we can show that $\sigma$ is
{\sl strictly}\/ positive almost everywhere and that $\ln \sigma\in L^1(\Omega)$
at the time $t>0$, provided that the same holds at the initial time. 
Working on the renormalized version
of \eqref{eq:2} and using a suitable uniform integrability argument,
we can actually prove strong $L^1$- (hence pointwise) convergence of $\ln\sigma$,
which is a key step in order to identify the cross-diffusion term in
the limit. It is worth noting that, due to the supercritical behavior
of that term, even for very smooth initial data, we 
do not expect that additional regularity (or parabolic smoothing) properties
for $\sigma$ could be proved. We also remark that the notion of weak solution considered here
is complemented by an ``integrated'' entropy inequality 
(cf.~\eqref{weakentropy} below), 
which is crucial for the sake of proving our weak-strong uniqueness result. 

\medskip

\noindent
{\bf Plan of the paper.} In Section~\ref{sec:main}, we introduce some notation and
state our three main results: existence of (global in time) weak solutions, 
existence of (local in time) strong solutions and weak-strong uniqueness.
Then, in Section~\ref{sec:pre} we prove some preliminary lemmas which are needed in
order to clarify the concept of weak solution, especially for what concerns regularity
properties and validity of the  energy and entropy inequalities.
Then, existence of weak solutions is proved in Section~\ref{sec:global} by 
using the system with subcritical sensitivity as an approximation of our model, deriving some 
a-priori estimates uniform with respect to the regularization parameter, 
and then passing to the limit by means of compactness arguments. Next, in 
Section~\ref{sec:local} we prove existence of local-in-time strong 
solutions by obtaining a further set of a-priori regularity estimates and 
using a local version of Gr\"onwall's lemma. Finally, in the last Section~\ref{sec:weakstrong} 
the weak-strong uniqueness result is proved by deriving a suitable relative energy inequality,
entailing that a weak and a strong solution originating from the same initial datum must 
coincide on the largest time interval in which they simultaneously exist.

%


\section{Assumptions and main results}
\label{sec:main}


\subsection{Notation}

We start introducing a set of notation which will be useful in
order to rigorously formulate our mathematical results. 
Letting $\Omega$ be a smooth bounded domain of $\RR^d$, $d\in\{2,3\}$, 
we set $H := L^2(\Omega)$ and $V := H^1(\Omega)$.
We will generally write $H$ in place of $H^d$ (with similar notation for other spaces), 
whenever vector-valued functions are considered. We denote by $(\cdot,\cdot)$ the
standard scalar product of~$H$ and by $\| \cdot \|$ the associated Hilbert norm.  
Moreover, we equip $V$ with the usual norm $\|\cdot\|_V^2 = \|\cdot\|^2 + \|\nabla \cdot\|^2$.
Identifying $H$ with its dual space $H'$ by means of the scalar
product introduced above, we obtain the chain of continuous
and dense embeddings $V\subset H \subset V'$.
We indicate by $\duav{\cdot,\cdot}$ the duality pairing between $V'$ and $V$, 
or, more generally, between $X'$ and $X$, where $X$ is some Banach space continuously
and densely embedded into $H$. Next, for a given Banach space $X$, the space
$\C_w([0,T];X )$ consists of the functions on $[0,T]$ taking values in $X$ 
that are continuous with respect to the weak topology of $X$.

The total variation of a function $E:[0,\infty)\ra \R$ is given by 
$$ 
  \TV(E,[0,T]):= \sup_{0=t_0<\ldots <t_n=T} \sum_{k=1}^n \lvert E(t_{k-1})-E(t_k) \rvert, 
$$
where the supremum is taken over all finite partitions of the interval $[0,T]$. 
We denote the space of all bounded functions of bounded total variations on $[0,T]$ 
by~$\Bv([0,T])$. We recall that the left- and right-hand limit of a function $E\in\Bv([0,T])$ 
are well defined in any point in time $t\in[0,T]$ and will be denoted by $ E(t-)$ and $E(t+)$, respectively. 
Similarly, we will denote by $\Bv([0,T];X)$ the Banach space of the functions that have
bounded total variation with respect to the norm of the generic Banach space $X$.
By $\mathcal{M}^+(\ov\Omega \times [0,T])$, we denote the space of nonnegative 
Borel measures  on $\ov\Omega \times [0,T]$. Finally,
for a real number $a\in \R$, $a_+$ and $a_-$ will denote its non-negative 
and non-positive parts, respectively.


\subsection{Weak solutions}
\label{subsec:weak}

Our main assumptions on the coefficients of the problem can be summarized as follows:
\begin{hypothesis}\label{hypo:weak}
    We assume $\Omega $ to be a regular bounded domain in $\R^d$, for $d\in \{2,3\}$;
    moreover, we assume  the chemotactic response coefficient $\chi$ to be strictly 
    positive and the function $\alpha : \R^2 \to \R$ to be Lipschitz continuous
    and to satisfy the uniform boundedness condition   
    $\underline\alpha \leq \alpha(r,s) \leq \ov \alpha$ for every $(r,s)\in \R^2$ 
    and for some constants (not necessarily positive) $\underline\alpha \leq \ov \alpha$.
    Furthermore, we assume the potential $F$ to be given by~\eqref{Flog} with 
    $\lambda \geq 0$. For later convenience, we denote as $\beta$
    the ``monotone part'' of $F'$ (or, more precisely, of the 
    {\sl subdifferential}\ $\partial F$). Namely, we set
    \begin{equation}\label{defi:beta}
        \beta(r) = \ln(1+r) - \ln(1-r), \quad r\in(-1,1),
    \end{equation}
    in such a way that $F'(r)=\beta(r)-\lambda r$.
\end{hypothesis}
\noindent%
It is also worth detailing here our assumptions on the initial data:
\begin{hypothesis}\label{hypo:data}
 We assume the following conditions:
 \begin{align}\label{hp:fhi0}
   & \fhi_0 \in V, \qquad F(\fhi_0) \in L^1(\Omega), 
    \qquad \frac1{|\Omega|}\io \fhi_0\,\de x \in (-1,1),\\
  \label{hp:sigma0}
   & \sigma_0 \in L^1(\Omega), \quad \sigma_0>0~~\text{a.e.\ in }\,\Omega, 
    \quad \sigma_0 \ln (1 + \sigma_0) \in L^1(\Omega),\\
  \label{hp:sigma00}
   & \ln \sigma_0\in L^1(\Omega),
     \quad (1 + (\ln\sigma_0)_-) \ln (1 + (\ln \sigma_0)_-) 
      \in L^1(\Omega).
 \end{align}
\end{hypothesis}
\noindent%
\begin{remark}\label{rem:s0}
As the hypotheses on $\sigma_0$ are a bit technical, some words of explanation
are in order. Basically, \eqref{hp:sigma0} specifies the ``behavior at infinity''
of $\sigma_0$, which has to be ``slightly better than $L^1$''. We will actually
prove that this uniform integrability property is maintained in time and 
is very useful to get strong $L^1$-convergence. Relation \eqref{hp:sigma00}
has a similar role and in a sense specifies the ``behavior at $0$'' of $\sigma_0$. 
Actually, as our entropy formulation is written in terms of the auxiliary variable
$\ln\sigma$, condition \eqref{hp:sigma00}, which is also maintained in
time, ensures some uniform integrability of $\ln\sigma$, which plays a role
when we take the limit in the entropy relation.
\end{remark}
\begin{remark}\label{rem:alpha}
 Considering the behavior of $\alpha$ for large $\sigma$, it may be worth observing that, 
 in our results, we only need that the {\sl positive} part of $\alpha$ stays bounded; indeed, 
 a fast growth at infinity of the negative part of $\alpha$ (e.g.~as in the logistic
 case \eqref{ex:alpha}) would have a stabilizing effect and guarantee additional
 a-priori summability. As that case has already been considered in detail 
 in \cite{GHW,RSS}, we preferred to assume directly $\alpha$ bounded, 
 which reduces a bit the technical complications in the proofs.
\end{remark}
\noindent%
Recalling (cf.~\eqref{freeenergy}) the energy functional associated 
to system \eqref{eq}, i.e.,
\begin{align}\label{energy}
  \mathcal{E}(\varphi , \sigma ) := \int_\Omega \frac{1}{2}| \nabla \varphi|^2 
    + F (\varphi) +\chi  (1-\varphi)\sigma + \sigma (\ln \sigma - 1) \, \de x,
\end{align}
we can now detail our notion of weak solution to system \eqref{eq}:
\begin{definition}\label{def:weak}
A triple $(\varphi,\mu,\sigma)$ defined over $\Omega\times(0,T)$ 
is called a weak solution to 
the system~\eqref{eq}, if it satisfies the regularity properties 
\begin{subequations}\label{weak:reg}
\begin{align}
  & \mu \in L^2(0,T;V), \label{reg:mu}\\
  & \varphi (\cdot,\cdot) \in (-1,1) \text{\ \ and }
  \sigma (\cdot,\cdot) >0 \text{\ \ a.e.~in }\Omega \times (0,T), \label{reg:phi0}\\
  & \varphi \in \C_w([0,T];V) \cap H^1(0,T;V'), \label{reg:phi}\\
  & 
  \Delta\varphi \in L^1(0,T;L^1(\Omega)), %
  \label{reg:Dphi}\\
  & \ln \sigma \in L^\infty(0,T;L^1(\Omega)) 
  \cap L^2(0,T;V) 
       \cap \mathrm{BV} (0,T; (W^{1,p}(\Omega))^*) \ \text{ where }p > d, \label{red:ln}\\
  & \sigma \in \C_w{([0,T];L^1(\Omega))} \cap H^1(0,T;(W^{2,p}(\Omega))^*)
    \ \text{ where }\, p > d,\label{reg:timesig}\\
  & \sigma (\ln \sigma +1) \in L^\infty( 0,T ; L^1(\Omega)), \label{reg:sig}\\ 
  & \beta(\varphi) \ln (1 + |\beta(\varphi)| )  \in L^1(0,T;L^1(\Omega)), 
   \qquad 
   \beta(\fhi) \in L^2(0,T;L^1(\Omega)), 
   \label{reg:Fprime}\\
 \label{reg:nonl}
  & \sqrt{\sigma } \nabla ( \ln \sigma + \chi (1-\varphi)) \in L^2(0,T;H),
\end{align}
\end{subequations}
and the weak formulations 
\begin{subequations}\label{weak:def}
\begin{align}
  \langle \t  \varphi ,  \psi \rangle\,   
  +  \int_\Omega \nabla \mu \cdot \nabla \psi \, \de x \,  & = 0 && \text{a.e.~in }(0,T) ,\label{weak:phi}\\
 \langle \t \sigma , \vartheta \rangle 
   + \int_\Omega \sigma \nabla ( \ln \sigma + \chi ( 1- \varphi)) \cdot \nabla \vartheta \,\de x \, 
   & = \int_\Omega  \alpha (\varphi,\sigma) \sigma \vartheta \, \de x && \text{a.e.~in }(0,T) \label{weak:sig}
\end{align}
are fulfilled for any $\psi \in V$ and 
$\vartheta \in W^{2,p}( \Omega)$, $p>d$,
together with the pointwise relation 
\begin{align}\label{weak:mu}
  \mu = - \Delta \varphi + F' ( \varphi) - \chi \sigma \quad 
   \text{a.e.~in }\,\Omega \times (0,T) 
\end{align}
and the additional boundary condition (in the sense of traces on $\partial\Omega$)
\begin{align}
  \dn \fhi = 0 \quad \text{a.e.~in }\, (0,T). \label{bou:fhi}
\end{align}
 Moreover, there exists a computable 
number  $\zeta\in[0,\infty) $,  depending only on the assigned 
data of the problem (and in particular on the initial data), such that 
the entropy inequality holds in the following sense:
\begin{multline}\label{weakentropy}
     - \langle \ln \sigma , \theta \rangle \big|_0^T  
      + \int_0^T \int_\Omega \theta | \nabla \ln \sigma|^2
          - \nabla \ln \sigma \cdot \nabla \theta 
          - \chi \theta \nabla\ln \sigma \cdot \nabla \varphi 
          + \chi \nabla \varphi \cdot \nabla \theta \, \de x \, \de t\\
    + \| \theta_-\|_{\C(\ov\Omega\times [0,T])} 
          \left[ \int_0^T \int_\Omega| \nabla \ln \sigma |^2 \, \de x\, \de t 
             - \zeta \right]   
   +\int_0^T\int_\Omega \alpha (\varphi , \sigma) \theta 
   + \ln \sigma \t \theta \, \de x \, \de t \leq 0,
\end{multline}%
 for every $\theta \in \C^1([0,T];W^{1,p}(\Omega))$, where $p>d$. 
  Similarly,  the energy inequality holds in the following sense: 
there exists a further computable and bounded
number $Z  \geq 0$, whose expression
depends only on the assigned data of the problem (and in particular on the initial data),
such that, for every $t\in[0,T]$, there holds
\begin{align}\nonumber
  & \calE(t) 
   + \int_0^t \io \sigma \big| \nabla ( \ln \sigma  + \chi (1-\fhi)) \big|^2
     + | \nabla \mu|^2 \,\de x \, \de s  \le \calE(0)\\
 \label{enin}
  & + \int_0^t \io \alpha(\fhi,\sigma) \big( \sigma \ln\sigma + \chi\sigma( 1 - \fhi) \big) \, \de x \, \de s
  + {\ov \alpha} \Big[ Z - \int_0^t \io \big( \sigma (\ln\sigma - 1) + 1 \big) \,\de x \, \de s \Big],
\end{align}
where the quantity in square brackets is omitted if $\ov\alpha \le 0$ 
(i.e., if $\alpha$ is nonpositive). 
Finally, the initial values $ (\varphi(0),\sigma(0))= ( \varphi_0,\sigma_0)$ 
are attained in the sense of traces. 
\end{subequations}
\end{definition}
\begin{remark}[Generalized formulation]\label{rem:weakstrong}
  We note that, in order to state a ``standard'' weak formulation,
  it would be  sufficient to consider relations~\eqref{weak:phi}-\eqref{bou:fhi} only.
  On the other hand, in Definition~\ref{def:weak} we also added the logarithmic inequality~\eqref{weakentropy} 
  as this will be used in an essential way in the proof of the weak-strong uniqueness result. 
  That the weak-strong uniqueness result relies on a logarithmic inequality is not uncommon 
  in the literature: we may refer, for instance, to a thermodynamically consistent phase-field 
  model~\cite{ourpreviousone}, or to a predator-prey model with cross diffusion~\cite{Luisa}.
  We also point out that the logarithmic inequality is formulated in a somehow ``non-standard'' 
  way, where in particular we wanted to consider test functions $\theta$ with no sign properties. 
  We will further comment on this aspect in Remark~\ref{rem:login} below. We also observe
  that, with a straighforward modification in the proof, it can be shown that the analogue
  of \eqref{weakentropy} holds on every subinterval $(0,t)$ (and not just on $(0,T)$). 
\end{remark}
\begin{theorem}[Global existence of weak solutions]\label{thm:global}
 Let Hypotheses~\ref{hypo:weak}, \ref{hypo:data} be satisfied and let
 $T > 0$ be an assigned reference time (of arbitrary magnitude). Then, 
 there exists at least one weak solution according 
 to Definition~\ref{def:weak} defined over $\Omega\times (0,T)$.
\end{theorem}
%
\noindent
Theorem~\ref{thm:global} will be proven in Section~\ref{sec:global}.
\begin{remark}[Pointwise identification in $L^1(\Omega)$]\label{rem:pointwise}
    We note that in the above Definition~\ref{def:weak}, the novel \textit{a priori} estimate~\eqref{st:31c} leading to the
    regularity $ |\ln \sigma| \ln (1+ | \ln \sigma|) \in L^\infty(0,T;L^1(\Omega))$ 
    is essential in order to deduce the first regularity 
    in~\eqref{red:ln}.   More precisely, we may observe that 
    the equiintegrability of $\ln \sigma(t)$ for $t\in[0,T]$    
    implies the existence of left and right limits (in the weak
    $L^1$-sense) at any $t\in [0,T]$ thanks  to the additional regularity 
    $\ln \sigma \in\Bv(0,T;(W^{1,p}(\Omega))^*)$. 
    Indeed, consider a sequence  $ \{ t_n\} _{n\in\N}\subset [t_0,T]$, 
    such that $ t_n \searrow t_0$. Due to the $\Bv(0,T;(W^{1,p}(\Omega))^*)$
    regularity, we know that 
    $$
          \langle \ln \sigma(t_n) , \theta\rangle \to \langle \ln \sigma(t_0+) , 
             \theta \rangle \quad \text{for all } \theta \in W^{1,p}(\Omega) 
    $$
    as $ t_n \searrow t_0$. 
    Then, the boundedness of
    $ \sup_{n\in\N}\int_{\Omega}  |\ln \sigma(t_n) |  
       \ln (1+ | \ln \sigma(t_n)|) \, \de x$ permits us to deduce that 
       there exists $\ell_{t_0+}\in L^1(\Omega)$ 
    such that 
    \begin{equation}
        \ln \sigma(t_n) \rightharpoonup \ell_{t_0+} \quad \text{in }L^1(\Omega)\,
    \end{equation}
     at least for a subsequence of $\{t_n\}\searrow t_0$. Then, the uniqueness 
     of the weak limit   implies  that $\ell_{t_0+} = \ln \sigma(t_0+) $. 
    The same holds, of course, for the left limit as $t_n\nearrow t_0-$. 
      Consequently, we can identify the limits  $ \ln \sigma (t_0\pm)$ 
    also as a function in $L^1(\Omega)$ and not only as a distribution 
    in $ (W^{1,p}(\Omega))^*$. 
\end{remark}
\noindent%
We now move to considering ``strong'' solutions. More precisely,
assuming that the initial data satisfy the
additional regularity conditions
\begin{align}\label{init:reg1}
  & \fhi_0 \in H^{2}_{\bn}(\Omega), \qquad
   - \Delta \fhi_0 + f(\fhi_0) \in V,\\
  \label{init:reg2}
  & \sigma_0 \in H^{2}_{\bn}(\Omega),
\end{align}
in Section~\ref{sec:local} we shall prove the following 
local existence result:
\begin{theorem}[Local existence of strong solutions]\label{thm:local}
 Let the assumptions of\/ {\rm Theorem~\ref{thm:global}} hold and let,
 in addition, \eqref{init:reg1}-\eqref{init:reg2} be verified. Then, there exist
 a time $T_0\in (0,T]$, depending on the initial data as well as on the 
 other assigned parameters of the system,
 and a triple $(\fhi,\mu,\sigma)$ enjoying the regularity
 conditions
 \begin{align}\label{strong:fhi}
  & \fhi \in W^{1,\infty}(0,T_0;V') \cap H^{1}(0,T_0;V)
   \cap L^\infty(0,T_0;W^{2,6}(\Omega)),\\
  \label{strong:Ffhi}
  & \beta(\fhi) \in L^\infty(0,T_0;L^{6}(\Omega)),\\
  \label{strong:mu}
  & \mu \in L^\infty(0,T_0;V),\\
  \label{strong:sigma}
  & \sigma \in W^{1,4}(0,T_0;L^3(\Omega)) 
   \cap L^4(0,T_0;W^{2,3}(\Omega)),\\
  \label{strong:sigma2}
  & \sigma \in L^\infty(0,T_0;L^\infty(\Omega)),
 \end{align}
 and satisfying relations\/
 \eqref{eq:1}-\eqref{eq:2} in the pointwise sense (almost everywhere in 
 $\Omega\times (0,T_0)$), with the boundary conditions\/ \eqref{bound:1c} and the initial
 conditions\/ \eqref{init:1d} also holding almost everywhere. 
\end{theorem}
\noindent%

\noindent
A triple $(\fhi,\mu,\sigma)$ satisfying the conditions 
of Theorem~\ref{thm:local} will be called a ``strong solution'' in the sequel. 
We also point out that the exponents~$6$ in \eqref{strong:fhi} and 
\eqref{strong:Ffhi} originate from the use of 
three-dimensional Sobolev embeddings and 
could be improved to any $P\in[1,\infty)$ for $d=2$. 
As is customary
in the mathematical literature devoted to weak-strong uniqueness,
in the next statement we will denote by $(\tp,\tmu,\ts)$ a local strong solution
in the sense of the previous theorem and by $(\fhi,\mu,\sigma)$ a weak
solution in the sense of Theorem~\ref{thm:global}. We can then prove
\begin{theorem}[Weak-strong uniqueness]\label{thm:weakstrong}
  Let Hypothesis~\ref{hypo:weak} be satisfied, let the initial
  data also satisfy\/ \eqref{hp:fhi0}-\eqref{hp:sigma0},   
  and let $\alpha$ additionally fulfill
  \begin{align}\label{ass:alpha}
    |\alpha (x_1, x_2) - \alpha(y_1,y_2)| \leq C ( | x_1-y_1|+ | \sqrt{x_2}-\sqrt{y_2}|) 
     \quad \text{for all } (x_1,x_2), (y_1,y_2)\in \R^2.
 \end{align}
 Moreover, let $(\varphi , \mu, \sigma)$ be a weak solution defined over $(0,T)$
 according to Definition~\ref{def:weak} 
 and $(\tp, \tmu, \ts)$ be a strong solution according to Theorem~\ref{thm:local},
 defined over $(0,T_0)$ for some $T_0\in(0,T]$, with 
 $ (\varphi(0),\sigma(0))= (\tp(0),\ts(0))$. Then, it holds 
 $$
    (\varphi,\sigma)\equiv (\tp,\ts)\qquad \text{a.e.~on }\, \Omega \times [0,T_0) .   
 $$
\end{theorem}
\noindent
The above result will be shown in Section~\ref{sec:weakstrong} below. 
Here, we just observe that it will be obtained as a direct consequence
of a slightly more general result (Corollary~\ref{cor:weakstrong} below). 
 Notice also that, as $\beta$ is assumed to be a single-valued mapping,
comparing values in the second of \eqref{eq:1} one deduces 
that $\mu=\tmu$ as well.

The following elementary result is aimed at clarifying the 
``square-root Lipschitz continuity'' assumption \eqref{ass:alpha},
which is shown to hold provided that $\alpha$ is smooth 
and ``asymptotically constant'' (with a suitable convergence rate)
for large $\sigma$:
\begin{lemma}[Square-root Lipschitz continuity]\label{lem:sqrt_lipschitz}
 Let $\alpha \in C^1([0,\infty))$ with the derivative $ \alpha'$ being bounded
 and such that $ | \alpha'(r) | \leq \frac{c}{1+\sqrt{r}} $ for all $ r\geq 0 $.
 Then, there exists a constant $C>0$ such that, for all $\theta,\eta \ge 0$,
 \[
   |\alpha(\theta)-\alpha(\eta)| \le C |\sqrt{\theta}-\sqrt{\eta}|.
 \]
 Moreover, there holds
 that $ | \alpha(\theta) | \leq C (1 + \sqrt{\theta}) $ for all $ \theta \geq 0$.
\end{lemma}
\begin{proof}
  Assuming, with no loss of generality, $\eta<\theta$, applying  
Taylor's formula   to the function $r\mapsto \kappa(r):= \alpha(r^2)$, 
we observe that there exists a point 
  $\xi \in (\sqrt{\eta},\sqrt{\theta})$  such that 
$$
  \alpha(\theta)-\alpha(\eta) 
     = \kappa(\sqrt{\theta}) - \kappa(\sqrt{\eta}) 
    = 2 \xi \alpha'(\xi^2)  (\sqrt{\theta}-\sqrt{\eta}).
$$
Then, taking the absolute values and inserting the bound on the derivative of $\alpha $,
we find that there exists a constant $C>0$ such that
  $\sup_{\xi\geq 0}|2\xi\alpha'(\xi^2)| \leq   C$, 
which provides the first assertion. 
Then, the second one follows simply by choosing $\eta =0$. 
\end{proof}



\section{Preliminaries}\label{sec:pre}

We start with stating a couple of preliminary lemmas whose proofs are given,
respectively, in~\cite[Lem.~7.7]{Luisa} and in~\cite[Lem.~1]{envar} 
or~\cite[Lem.~2.11]{hyper}.
\begin{lemma}[An application of the Fenchel--Young inequality]\label{lem:fenchel_young}
	Let $r > 0$ and $u, \tilde{u}, w, \tilde{w} \in \R$
	such that $ |w|,|\tilde{w}|<r$ and $\tilde{u}, u > 0$ holds.
	We then have
	\begin{equation}\label{eq:est_via_fenchel}
		(w- \tilde{w}) (u - \tilde{u}) 
		\leq 
			\max \left \{ \frac{1}{4r}, 4r \right \} \left( \tilde{u} \, |w - \tilde{w}|^2 + u - \tilde{u} - \tilde{u}(\ln u - \ln \tilde{u}) \right).
	\end{equation}
\end{lemma}
Actually, the above lemma is stated in~\cite[Lem.~7.7]{Luisa} only for nonnegative $ w, \tilde{w} $,
but in fact this property is not essential in the proof and can be omitted.
\begin{lemma}\label{lem:invar}
 Let $f\in L^1(0,T)$, $g\in L^\infty(0,T) $ and $g_0\in\R$.
 Then the following two statements are equivalent:
 \begin{enumerate}
 \item
 The inequality 
 \begin{equation}
  -\int_0^T\phi'(\tau) g(\tau) \, \de \tau  
   + \int_0^T \phi(\tau) f(\tau) \, \de \tau - \phi(0)g_0 \leq 0 
  \label{ineq1}
 \end{equation}
 holds for all $\phi \in {\C}^1_c ([0,T))$ with $\phi \geq 0$.
 \item
 The inequality
 \begin{equation}
    g(t) -g(s) + \int_s^t f(\tau)\, \de \tau \leq 0 
    \label{ineq2}
 \end{equation}
 holds for a.e.~$s< t\in[0,T]$,
 including $s=0$ if we replace $g(0)$ with $g_0$.
 \end{enumerate}
 Moreover, if either of these conditions is satisfied, 
 then $g$ can be identified with a function in 
   $\Bv([0,T])$   such that
 \begin{equation}
    g(t+) - g(s-) + \int_s^t f(\tau) \, \de \tau \leq 0, 
    \label{ineq.pw}
 \end{equation}
 for all $s\leq t\in[0,T]$, where we have set $g(0-):= g_0$.
 In particular, it holds $g(0+)\leq g_0$ and $g(t+)\leq g(t-)$
 for all $t\in[0,T]$. 
\end{lemma}
See~\cite[Lem.~1]{envar} for a proof. 
\begin{lemma}\label{lem:meas}
Let $(\varphi,\sigma)$ fulfill the regularity requirements 
of~\eqref{weak:reg}. Then the following two formulations are equivalent:
\begin{itemize}
    \item [(i)] 
    there exists $\zeta \geq 0 $ with
    $\zeta \geq \int_0^T \int_\Omega |\nabla \ln \sigma |^2 \, \de x \, \de t$ 
    such that the inequality
    \begin{multline}\label{weakentropy3}
        -\langle \ln \sigma ,\theta \rangle \Big |_0^T 
        + \int_0^T \int_\Omega \theta | \nabla \ln \sigma|^2 
           - \nabla \ln \sigma \cdot \nabla \theta 
           - \chi \theta \nabla\ln \sigma \cdot \nabla \varphi 
           + \chi \nabla \varphi \cdot \nabla \theta \, \de x \,\de t\\
      + \| \theta_-\|_{\C(\ov\Omega\times [0,T])}
         \left[ \int_0^T \int_\Omega| \nabla \ln \sigma |^2 \, \de x \, \de t - \zeta \right]
        + \int_0^T \int_\Omega \alpha \theta + \ln \sigma \t \theta \, \de x \, \de t \leq 0 
\end{multline}
holds for all $ \theta \in \C^1([0,T];W^{1,p}(\Omega))$ with $p>d$.
\item[(ii)]
 there exists  a measure~$\xi\in \mathcal{M}^+(\ov\Omega \times [0,T])$ such that
 \begin{multline}\label{entromeas}
      - \langle \ln \sigma ,\theta \rangle \Big |_0^T 
      + \int_0^T \int_\Omega \theta | \nabla \ln \sigma|^2 
         - \nabla \ln \sigma \cdot \nabla \theta 
         - \chi \theta \nabla\ln \sigma \cdot \nabla \varphi
         + \chi \nabla \varphi \cdot \nabla \theta \, \de x \, \de t\\
      + \int_0^T\int_{\ov\Omega} \theta \,\de \xi  (t,x)  
      + \int_0^T\int_\Omega \alpha \theta +\ln \sigma \t \theta \, \de x\, \de t = 0,
\end{multline}
for all $ \theta \in \C^1( [0,T];W^{1,p}(\Omega))$ for $p>d$; moreover, there also holds
\begin{equation}\label{measineq}
    -\int_0^T \int_\Omega \theta \, \de \xi (t,x)
    \leq \| \theta_-\|_{L^\infty(\Omega\times (0,T))}
      \left[ \zeta- \int_0^T \int_\Omega| \nabla \ln \sigma |^2 \, \de x \,\de t  \right] 
\end{equation}  
for all $ \theta \in \C( \ov\Omega \times [0,T];\R)$. 
\end{itemize}
Moreover, if either (i) or (ii) holds, then there also holds the 
point-wise in time formulation 
 \begin{multline}\label{entropw}
      -  \langle \t \ln \sigma , \theta \rangle
      +  \int_\Omega \theta | \nabla \ln \sigma|^2 
         - \nabla \ln \sigma \cdot \nabla \theta 
         - \chi \theta \nabla\ln \sigma \cdot \nabla \varphi
         + \chi \nabla \varphi \cdot \nabla \theta \, \de x \\
      + \int_\Omega \alpha \theta  \, \de x  = 0, 
\end{multline}
a.e.~in $(0,T)$ and for all  $\theta \in \C^1([0,T];W^{1,p}(\Omega))$ with $p>d$. 
\end{lemma}
\begin{proof}
 First, we prove that (i) $\Rightarrow$ (ii). Therefore, we consider the linear mapping 
 $ \f l : \C^1( [0,T];W^{1,p}(\Omega)) \to \R $ for $p>d$ given by 
 \begin{align*}
     \langle \f l, \theta \rangle :=  &-\int_\Omega \ln \sigma \theta \, \de x \Big |_0^T 
        + \int_0^T \int_\Omega \theta | \nabla \ln \sigma|^2 
        - \nabla \ln \sigma \cdot \nabla \theta 
        - \chi \theta \nabla\ln \sigma \cdot \nabla \varphi\, \de x \\ 
     & +\int_{\Omega} \chi \nabla \varphi \cdot \nabla \theta 
         + \alpha (\varphi,\sigma) \theta 
         + \ln \sigma \t \theta \, \de x \, \de t \,.
 \end{align*}
 This is a linear and continuous mapping on     $\C^1( [0,T];W^{1,p}(\Omega))$. 
 Moreover, we define 
 $ p : \C( \ov\Omega \times [0,T];\R)\to \R$ via 
 $$
 p(\theta):= \| \theta_-\|_{\C(\ov\Omega\times [0,T])}\left[ \zeta
    -\int_0^T \int_\Omega| \nabla \ln \sigma |^2 \, \de x \, \de t \right].
 $$ 
 Then, due to~\eqref{weakentropy3}, it holds that 
 $ \langle \f l , \theta\rangle \leq p(\theta)$ for 
 all $\theta\in \C^1( [0,T];W^{1,p}(\Omega))$.
 
 Hahn--Banach's theorem~\cite[Thm.~1.1]{brezis} guarantees that the 
 functional $\f l$ can be extended to a continuous linear mapping on 
 $\tilde{\f l }:\C( \ov\Omega \times [0,T];\R)\to \R $
 such that $ \langle \tilde{\f l}, \theta \rangle \leq p(\theta)$
 for all $\C( \ov\Omega \times [0,T];\R)$ and 
 $ \langle \f l ,\theta \rangle = \langle \tilde{\f l}, \theta \rangle $ 
 for all $\C^1( [0,T];W^{1,p}(\Omega))$. 
 Via Riesz' representation theorem, such a linear functional can be interpreted 
 as a measure, \textit{i.e.,} there exists $ \xi\in \mathcal M(\ov\Omega\times [0,T])$ 
 such that 
 $\langle \tilde{\f l} , \theta \rangle= - \int_0^T\int_{\Omega} \theta\,\de \xi(x,t)$. 
 The properties of $\tilde{\f l} $ imply the equations~\eqref{entromeas} and the 
 inequality~\eqref{measineq}. Choosing $ \theta \in \C(\ov\Omega\times [0,T]; \R)$
 with $ \theta_- =0$ implies that 
 $\int_0^T\int_{\ov\Omega} \theta \, \de \xi  (t,x) \geq 0 $
 for all $\theta \in \C(\ov\Omega\times [0,T];[0,\infty))$,
 which is nothing else than  $\xi\in \mathcal{M}^+(\ov\Omega \times [0,T])$. 
 
 \smallskip
 
 Next, the implication (ii)$\Rightarrow$(i) can be obtained by adding the 
 inequality~\eqref{measineq} to~\eqref{entromeas}. 
 
 Finally, by the usual definition of the weak time derivative 
 $$  - \langle \t \ln \sigma , \theta \rangle 
  =  - \langle \ln \sigma, \theta \rangle \Big |_0^T
    + \int_0^T\int_\Omega \ln \sigma \t \theta \, \de x\, \de t,
 $$
 we observe the regularity of the weak time derivative by comparison in~\eqref{entromeas}
 and conclude that the resulting relation also holds point-wise a.e.~in time, which concludes the proof.
\end{proof}
\begin{remark}[Logarithmic inequality]\label{rem:login}
 As noted in Remark~\ref{rem:weakstrong}, we enriched the usual weak formulation by a
 logarithmic inequality so to ensure weak-strong uniqueness. 
 But usually, this inequality is formulated in a weaker way, \textit{i.e.}, if we insert  
 $  \theta = \phi\vartheta$ for $\phi \in \C^1([0,T); [0,\infty))$ and 
 $ \vartheta \in \C^1( [0,T], W^{1,p}(\Omega))$ for $p>d$
with $\vartheta \geq0$ in~\eqref{weakentropy}, 
 we infer from Lemma~\ref{lem:invar} with $s=0$  that
  \begin{multline*}
       - \langle \ln \sigma ,\vartheta \rangle  \Big |_0^t 
          + \int_0^t \int_\Omega \vartheta | \nabla \ln \sigma|^2 
             - \nabla \ln \sigma \cdot \nabla \vartheta 
             - \chi \vartheta \nabla\ln \sigma \cdot \nabla \varphi 
             + \chi \nabla \varphi \cdot \nabla \vartheta \, \de x \, \de \tau \\  
    + \int_0^t\int_\Omega \alpha(\varphi,\sigma)\vartheta + \ln \sigma \t \vartheta \, \de x \, \de \tau
      \leq 0 
\end{multline*}
holds for   a.e.~$0 <  t \leq T$ 
with $\ln \sigma(0)= \ln \sigma_0$ and all $ \vartheta \in \C^1( [0,T], W^{1,p}(\Omega))$ for $p>d$
with $\vartheta \geq0$. 
From the Remark~\ref{rem:pointwise}, we even infer that we may understand $ \ln \sigma $ pointwise in time in $L^1(\Omega)$.   Moreover, setting $ \ln \sigma (t)= \ln \sigma (t+)$, we 
may  also deduce the ``standard'' form of the inequality:
\begin{multline}\label{en:new}
       - \int_{\Omega}
        \ln \sigma \vartheta \, \de x   \Big |_0^t 
          + \int_0^t \int_\Omega \vartheta | \nabla \ln \sigma|^2 
             - \nabla \ln \sigma \cdot \nabla \vartheta 
             - \chi \vartheta \nabla\ln \sigma \cdot \nabla \varphi 
             + \chi \nabla \varphi \cdot \nabla \vartheta \, \de x \, \de \tau \\  
    + \int_0^t\int_\Omega \alpha(\varphi,\sigma)\vartheta + \ln \sigma \t \vartheta \, \de x \, \de \tau
      \leq 0 
\end{multline}
for all $t\in [0,T]$ and all $ \vartheta \in \C^1( [0,T], W^{1,p}(\Omega))$ with $p>d$. 

Actually,   the formulation \eqref{en:new} of the  entropy inequality may  
suffice to prove weak-strong uniqueness of solutions, as we will
see in Section~\ref{sec:weakstrong}.   The advantage of adding the error term
depending on $\zeta$ (cf.~\eqref{weakentropy}) stands in the fact that, then,
the formulation also becomes weakly closed; namely, every cluster point 
of a sequence of solutions in the sense of Definition~\ref{def:weak} 
is also a solution, as we will prove in Proposition~\ref{prop:stab}, below. 
Note also that the energy inequality~\eqref{enin} is not required 
for proving weak-strong uniqueness;
however, it is crucial (also as a source of estimates) in the proof of 
weak sequential stability. As noted in the statement of 
Theorem~\ref{thm:giulio}, if $ \alpha \leq 0 $ a.e., 
the correction term depending on $Z$ in~\eqref{enin} can be omitted;
namely, a standard form of the energy inequality holds in that case. 
On the other hand, for general $\alpha$, the integrand
$\alpha(\fhi,\sigma)\sigma\ln\sigma$ cannot be managed by (lower) semicontinuity 
methods, and this is why the correction term appears. Without 
it, the energy inequality would also cease to be 
weakly sequentially stable (see Subsec.~\ref{subsec:enen} below 
 for more details). 
\end{remark}


\section{Global existence of weak solutions}
\label{sec:global}


\subsection{Proof of Theorem~\ref{thm:global}: regularized system}

In order to prove the existence of weak solutions in the sense of Definition~\ref{def:weak}, 
we take a regularization parameter $\epsi\in(0,1)$ intended to go to $0$ in the limit and
consider the following approximated system:
\begin{subequations}\label{eq:approx}
\begin{align}
 \t \rp - \Delta \rmu = 0,\qquad \rmu = - \Delta \rp + F'(\rp) - \chi \rs, 
 &\quad \text{in }\Omega \times (0,T),\label{eq:1approx}\\
 \t \rs - \Delta \rs - \chi \dive \Big ( \frac{\rs}{1+\varepsilon \rs }\nabla  ( 1-\rp) \Big)
  = \alpha(\rp,\rs) \rs,  &\quad \text{in }\Omega \times (0,T), \label{eq:2approx}
\end{align}
equipped with the boundary conditions 
\begin{align}\label{eq:3approx}
 \f n \cdot \nabla \rmu = 0 = \f n \cdot \nabla \rp, \qquad 
 \f n \cdot \Big( \nabla \rs
   +\frac{\chi \rs}{1+\varepsilon \rs }\nabla  ( 1-\rp)\Big) = 0,
   \quad \text{on }\partial\Omega \times (0,T)
\end{align}
and the initial conditions
\begin{align}\label{eq:4approx}
  \rp (0) = \varphi_0 ^\varepsilon,\qquad\sigma(0) =\sigma_0^\varepsilon, \quad \text{in } \Omega,
\end{align}
\end{subequations}
 where $\varphi_0^\varepsilon$ and $\sigma_0^\varepsilon$ are suitable regularizations
 of the initial data (more details will be given below).
System~\eqref{eq:approx} was addressed in the recent manuscript~\cite{GiulioNew}, 
where it was part of a more complicated model of {Cahn}-{Hilliard}-{Brinkman} 
type also accounting for macroscopic velocity effects. Neglecting here the 
velocity and adapting the statement to our notation, the result proved 
in \cite[Thm.~2.3]{GiulioNew} in the case $p=2$ (i.e., taking a bounded
sensitivity as in \eqref{eq:2approx}, \eqref{eq:3approx})
can be formulated as follows:
\begin{theorem}[Weak solutions to the regularized system]\label{thm:giulio}
 Let $F$ be given by \eqref{Flog} and let $\epsi\in(0,1)$.
 Let $\fhi_0^\epsi$, $\sigma_0^\epsi$ satisfy\/ 
 \eqref{hp:fhi0}, \eqref{hp:sigma0}, and,
 in addition, let
 \begin{equation}\label{sigma0l2}
   \sigma_0^\varepsilon \in H, \qquad
    \ln \sigma_0^\varepsilon \in L^1(\Omega).
 \end{equation}
 Then, there exists at least one triple $(\fhi\ee,\mu\ee,\sigma\ee)$ with the regularity
 properties
 \begin{align}\label{rego:fhi1}
  & \fhi\ee \in H^1(0,T;V')\cap L^\infty(0,T;V) 
    \cap L^{4}(0,T;H^2(\Omega)) \cap L^2(0,T,W^{2,6}(\Omega)),\\
 \label{rego:Ffhi1}
  & F(\fhi\ee) \in L^\infty(0,T;L^1(\Omega)),
   \qquad F'(\fhi\ee) \in L^2(0,T;L^{6}(\Omega)), \\
 \label{rego:mu1}
  & \mu\ee\in L^2(0,T;V) ,\\ 
 \label{rego:sigma1}
  & \sigma\ee \in H^1(0,T;V') \cap L^\infty(0,T;H) \cap L^2(0,T;V),\\
 \label{pos:sigma1}
  & \sigma\ee>0~~\text{a.e.\ in }\,\Omega\times(0,T), \qquad
   \ln \sigma\ee \in  L^\infty(0,T;L^1(\Omega))  \cap L^2(0,T;V),
 \end{align}
 and such that the following weak relations hold a.e.~in $(0,T)$ 
 for every $\psi,\vartheta\in V$:
 \begin{align}
  & \langle \t  \varphi\ee ,  \psi \rangle\,   
   + \int_\Omega \nabla \mu\ee \cdot \nabla \psi \, \de x \,  
    = 0,\label{weak:phie}\\
  & \langle \t \sigma\ee , \vartheta \rangle 
   + \int_\Omega \Big( \nabla \sigma\ee + \chi \frac{\rs}{1+\varepsilon \rs } \nabla (1- \fhi\ee) \Big) 
           \cdot \nabla \vartheta \,\de x
    = \int_\Omega  \alpha (\varphi\ee,\sigma\ee) \sigma\ee \vartheta \, \de x,
    \label{weak:sige}
 \end{align}
 together with the relations 
 \begin{equation}
    \rmu = - \Delta \rp + F'(\rp) - \chi \rs, \qquad 
     \dn \rp = 0,
 \end{equation}
 intended to hold pointwise, respectively in $\Omega\times(0,T)$ and 
 on $\partial\Omega\times(0,T)$, and the initial conditions \eqref{eq:4approx},
 intended pointwise in $\Omega$.
\end{theorem}
\noindent%
In the sequel, starting from the above existence theorem, we will let 
$\epsi\searrow 0$ so to get in the limit, up to taking a suitable subsequence,
a solution in the sense of Def.~\ref{def:weak}.
In order for this argument to work, we need however to 
reinforce a bit our assumptions on the regularized initial data 
and to specify how they behave with respect to the regularization
parameter $\epsi$. Concerning $\fhi_0^\varepsilon$, as 
the same condition \eqref{hp:fhi0} is required both in the statement
of Theorem~\ref{thm:global} and in that of Theorem~\ref{thm:giulio},
there is nothing to modify and we can simply assume $\fhi_0^\varepsilon\equiv \fhi_0$,
independently of $\epsi$.   Conversely,
in order to construct $\sigma_0^\varepsilon$,
we preliminarly notice that, comparing \eqref{sigma0l2} with 
Hypothesis~\ref{hypo:data}, two differences
arise: the additional 
summability required in \eqref{sigma0l2} and the lack of the second 
of \eqref{hp:sigma00} in the statement of Theorem~\ref{thm:giulio}.
As already observed, the latter condition is required to prove the additional summability of $\sigma$ close to $0$, a property
which was not observed in \cite{GiulioNew}, while the first of \eqref{sigma0l2} 
is necessary to obtain the parabolic
regularity \eqref{rego:sigma1} at the regularized level. Then,  
starting from $\sigma_0$ as in Hypothesis~\ref{hypo:data}
and taking $\sigma_0^\varepsilon = \min(\sigma_0,\epsi^{-1/2})$, it is readily seen 
that the resulting family $\{\sigma_0^\epsi\}$ complies with \eqref{sigma0l2}.
More precisely, by the dominated convergence theorem, one has
\begin{align}
  \label{sigmaepsi0}
   & \sigma_0^\varepsilon \to \sigma_0 \quad\text{strongly in }\,L^1(\Omega),\\  
  \label{sigmaepsi01}
   & (\ln \sigma_0^\epsi)_- =(\ln\sigma_0)_-
    \quad\text{for all }\, \epsi\in(0,1), \\  
  \label{sigmaepsi1}
   & \epsi^{1/2} \| \sigma_0^\varepsilon \| \le c,
\end{align}
with $c$ independent of $\epsi$. 
It is also worth observing that, due to the choice of the singular potential \eqref{Flog}, 
\eqref{rego:Ffhi1} implies in particular
\begin{equation}\label{trepunti}
  | \rp(x,t) | < 1 \quad\text{for a.e.~}\,(x,t)\in\Omega\times (0,T)
\end{equation}
and every $\epsi\in(0,1)$.
\begin{remark}\label{rem:rego}
 The regularity proved in the above theorem appears to be optimal for what regards $\fhi\ee$; 
 actually, for the Cahn-Hilliard equation with logarithmic potential \eqref{Flog}, 
 in three space dimensions, even for $C^\infty$ initial data, it is not known whether 
 \eqref{rego:fhi1}, \eqref{rego:Ffhi1} could be improved, at least as far as {\sl global
 in time}\/ estimates are looked for. On the other hand, as observed in \cite[Rem.~2.6]{GiulioNew}, 
 assuming better conditions (compared to \eqref{sigma0l2}) on the initial datum,
 it is rather simple to improve the regularity of $\sigma\ee$, provided that
 the regularity of $\fhi\ee$ (on which depends the cross-diffusion term)
 allows for that. Indeed, in the case of bounded sensitivity, 
 \eqref{eq:2approx} basically behaves as a standard parabolic equation
 of quasi-linear type. 
 We also remark that the regularity properties 
 \eqref{rego:fhi1}-\eqref{pos:sigma1} refer to the three-dimensional case; for $d=2$, 
 a slightly improved statement holds; indeed, as already noted, the 
 ``Sobolev'' exponents $6$ in \eqref{rego:fhi1},
 \eqref{rego:Ffhi1} in 2D may be replaced by
 any exponent $P\in[1,\infty)$.
\end{remark}


\subsection{Proof of Theorem~\ref{thm:global}: a priori estimates}
\label{subsec:apriori}

In this part we consider an $\epsi$-dependent family $(\fhi\ee,\mu\ee,\sigma\ee)$
of weak solutions to system \eqref{eq:approx}, as provided by Theorem~\ref{thm:giulio},
and derive multiple \textit{a priori} estimates, independent of $\epsi$,
in order to subsequently let $\epsi\searrow 0$ and attain a solution to the 
original system in the sense of Definition~\ref{def:weak}. 
We report the procedure in some detail since the argument provides
a good insight about the basic variational and structure properties of the system.

First of all, integrating~\eqref{eq:2approx} in space and time, using 
Hypothesis~\ref{hypo:weak} and the comparison principle for ODE's,
we immediately  observe an estimate of the mass of $\rs$, \textit{i.e.,}
\[ 
   e^{\underline{\alpha} t} \int_{\Omega}\sigma_0 \, \de x
   \le \int_{\Omega}\rs(t) \, \de x 
    \leq e^{\overline{\alpha} t} \int_{\Omega}\sigma_0 \, \de x 
      \quad \text{ for all }\,t\in[0,T].
\]
Analogously, integration of the first \eqref{eq:1approx} provides the mass
conservation property
\begin{equation}\label{massco}
  \int_{\Omega}\fhi\ee(t) \, \de x 
    =  \int_{\Omega}\fhi_0 \, \de x \quad \text{ for all }\,t\in[0,T].
\end{equation}
Additionally, we may observe that condition \eqref{hp:sigma0} and 
a simple argument based on Stampacchia's truncations method ensure that 
\begin{equation}\label{posi}
   \sigma\ee(x,t)\ge 0 \quad\text{for a.e.~}\,(x,t)\in \Omega\times(0,T).
\end{equation}
Next, in order to derive the energy-dissipation law, we multiply~\eqref{eq:2approx} by 
$\ln \rs + \varepsilon (\rs -1) + \chi (1-\rp) $, add~\eqref{eq:1approx}$_1$ 
multiplied by $\rmu$ as well as \eqref{eq:1approx}$_2$ multiplied by $\t \fhi\ee$,
and integrate over $\Omega$. Then, we note that \eqref{eq:2approx} can be
equivalently rewritten as 
\begin{equation}\label{sig:weak}
 \t \rs - \di \Big[ \frac{\sigma\ee}{1+\epsi\sigma\ee} \nabla \big( 
   \ln \sigma\ee + \epsi (\sigma\ee - 1) + \chi ( 1-\rp) \big) \Big]
  = \alpha(\rp,\rs)\rs,
\end{equation}
from where, by simple calculations, we arrive at the estimate
\begin{multline}\label{enregest}
   \ddt \int_{\Omega}\rs(\ln\rs-1) 
    + \frac{1}{2} |\nabla \rp|^2 + F(\rp)+\chi \rs (1-\rp)+ \frac{\varepsilon}{2}(\rs-1)^2 \,\de x \\
    + \io\frac{\rs}{1+\varepsilon\rs} \big| \nabla ( \ln \rs + \varepsilon(\rs -1)+ \chi (1-\rp)) \big|^2
     + | \nabla \rmu|^2 \,\de x \\
    \leq \int_\Omega \alpha(\rp,\rs)\rs \big( \ln \rs + \varepsilon (\rs-1)
       +\chi (1-\rp)\big)\,\de x.
\end{multline}
%
%
Accordingly with the above relation, we introduce the approximate energy as
\begin{equation}\label{defi:Eee}
  \mathcal{E}_\varepsilon( \varphi\ee, \sigma\ee):=  \int_{\Omega}\sigma\ee (\ln\sigma\ee -1) 
    + \frac{1}{2}|\nabla \varphi\ee  |^2 + F(\varphi\ee  )+\chi \sigma\ee  (1-\varphi\ee  )
    + \frac{\varepsilon}{2}(\sigma\ee -1)^2 \,\de x.
\end{equation}
Then, it is worth observing that there exist constants $\kappa>0$ and $C_0>0$, independent
of $\epsi$, such that the $\calE\ee$ enjoys the following uniform
coercivity property:
\begin{equation}\label{coerc:Eee}
  \mathcal{E}_\varepsilon( \varphi\ee, \sigma\ee) 
    \ge \kappa \Big( \| \sigma\ee \ln (1 + \sigma\ee) \|_{L^1(\Omega)}
       +  \| \fhi\ee \|_{V}^2
       + \| F(\fhi\ee) \|_{L^1(\Omega)}
       + \epsi \| \sigma\ee - 1 \|^2 \Big) - C_0.
\end{equation}
Note in particular that $C_0$ depends on the assigned parameters $\lambda$ and $\chi$.
With similar arguments one can also deduce the boundedness condition
\begin{equation}\label{bound:Eee}
  \mathcal{E}_\varepsilon( \varphi\ee, \sigma\ee) 
    \le c \Big( \| \sigma\ee \ln (1 + \sigma\ee) \|_{L^1(\Omega)}
       +  \| \fhi\ee \|_{V}^2
       + \| F(\fhi\ee) \|_{L^1(\Omega)}
       + \epsi \| \sigma\ee - 1 \|^2 + 1 \Big).
\end{equation}
Here and below, $c>0$ denotes a generic positive constant depending only on the 
assigned problem data and independent of the approximation parameter $\epsi$.
Note that \eqref{coerc:Eee} and \eqref{bound:Eee} are based on Hypothesis~\ref{hypo:weak}
and on relation \eqref{trepunti}, which, in turn, depends in an essential way 
on the fact that the logarithmic potential \eqref{Flog} is {\sl not}\/ regularized. 

Next, it is not difficult to prove the following control
of the right-hand side of \eqref{enregest}:
\begin{equation}\label{rhs:11}
  \int_\Omega \alpha(\rp,\rs)\rs \big( \ln \rs + \varepsilon (\rs-1)
       +\chi (1-\rp)\big) \,\de x
    \le c \big( 1 + \mathcal{E}\ee(\fhi\ee,\sigma\ee) \big).
\end{equation}
Hence, thanks to Gr\"onwall's lemma, 
\eqref{enregest} entails the following a priori bounds:
\begin{align}\label{st:11}
  & \| \fhi\ee \|_{L^\infty(0,T;V)} 
   + \| F(\fhi\ee) \|_{L^\infty(0,T;L^1(\Omega))} \le c,\\
 \label{st:12}
  & \| \sigma\ee \ln (1 + \sigma\ee) \|_{L^\infty(0,T;L^1(\Omega))} \le c,\\
 \label{st:13} 
  & \Big\| \frac{\rs^{1/2}}{(1+\varepsilon\rs)^{1/2}} 
        \nabla \big( \ln \rs + \varepsilon(\rs -1)+ \chi (1-\rp)\big ) \Big\|_{L^2(0,T;H)} \le c,\\
 \label{st:14} 
  & \| \nabla \mu\ee \|_{L^2(0,T;H)} \le c,\\
 \label{st:11new} 
  & \epsi^{1/2} \| \sigma\ee \|_{L^\infty(0,T;H)} \le c.
\end{align}
Note in particular that, to control the energy at $t=0$
uniformly in $\epsi$, we have used that the approximating initial 
data are designed so to satisfy \eqref{sigmaepsi0}-\eqref{sigmaepsi1}.

To proceed, we test the second \eqref{eq:1approx} by $\fhi\ee - m_0$,
where we have set $m_0:= |\Omega|^{-1}\io\fhi\ee\,\de x = |\Omega|^{-1}\io\fhi_0\,\de x$,
where the second inequality follows from  the mass conservation \eqref{massco} and the 
fact that $\fhi_0$ is taken independent of $\epsi$ (namely, we did not need
to regularize it). Then, we obtain
\begin{equation}\label{MZ:01}
  \| \nabla \fhi\ee \|^2  
   + \int_\Omega F'(\fhi\ee) ( \fhi\ee - m_0 )\, \de x
   = \int_\Omega ( \mu\ee + \chi\sigma\ee ) ( \fhi\ee - m_0 )\, \de x.
\end{equation}
Next, we observe that the following relation:
\begin{equation}\label{MZ:11}
  \int_\Omega F'(\fhi\ee) ( \fhi\ee - m_0 )\, \de x
   \ge \kappa \| \beta(\fhi\ee) \|_{L^1(\Omega)} - c,
\end{equation}
holds for suitable constants $\kappa>0$ and $c\ge 0$ independent of $\epsi$
 (but depending on the datum $m_0$),
where we recall that $\beta$ denotes the ``monotone part'' of $F'$.
Actually, to to prove the above, one can proceed as in \cite{MiranvilleM2AN2004}
to which we refer the reader for details. Concerning the last integral on the 
right-hand side of \eqref{MZ:01}, we first observe that, by the Poincar\'e-Wirtinger inequality, 
there follows 
\begin{equation}\label{MZ:02}
  \int_\Omega \mu\ee ( \fhi\ee - m_0 )\, \de x
 = \int_\Omega \Big( \mu\ee - |\Omega|^{-1}\io\mu\ee \,\de x \Big)
   ( \fhi\ee - m_0 )\, \de x
   \le c \| \nabla\mu\ee \| \| \nabla\fhi \ee \| 
   \le c \| \nabla\mu\ee \|,
\end{equation}
the last inequality following from \eqref{st:11}. Next, by \eqref{trepunti}
and \eqref{st:12},
\begin{equation}\label{MZ:03}
  \int_\Omega \chi\sigma\ee ( \fhi\ee - m_0 )\, \de x
   \le c \| \sigma\ee \|_{L^1(\Omega)} \le c.
\end{equation}
Replacing \eqref{MZ:11}-\eqref{MZ:03} into \eqref{MZ:01}
and squaring the resulting relation,  
it is then not difficult to deduce the additional bounds
(see also \cite[Sec.~3]{GiulioNew} for further details)
\begin{align}\label{st:15}
  & \| \beta(\fhi\ee) \|_{L^2(0,T;L^1(\Omega))} + \| \mu\ee \|_{L^2(0,T;V)} \le c,\\
 \label{st:16}
  & \| \t \fhi\ee \|_{L^2(0,T;V')} \le c.
\end{align} 
The bounds derived so far depend directly on the fundamental physical balance 
laws and basically coincide with those proved in \cite{GiulioNew}. The subsequent
part of the proof instead is totally new as we now look for estimates
in a weaker regularity setting, which corresponds, roughly speaking, to the 
Orlicz space usually denoted as $L\log L$ (see, e.g., \cite{Orlicz}). 

To this aim, we may test the second of \eqref{eq:1approx} by 
$\rho(\fhi\ee):=\ln(1 + |\beta(\fhi\ee)|) \sign(\fhi\ee)$,
which is an admissible test function because,
thanks to the second of \eqref{rego:Ffhi1}, $\rho(\fhi\ee)\in L^P(\Omega\times (0,T))$ for
every $P\in[1,\infty)$. Note also that
the second of \eqref{eq:1approx} holds (at least) 
as a relation in $L^2(0,T;H)$ and that, thanks to the fact $\beta(0)=0$, 
$\rho$ is a monotone function. Now, a well-known integration by parts formula 
for (possibly singular) monotone operators gives
\begin{equation}\label{giu:a1}
  - \io \rho(\fhi\ee) \Delta \fhi\ee \,\de x \ge 0.
\end{equation}
The above can be proved for instance replacing $\rho$ with its Yosida 
approximation (see, \cite{barbu,brezis}), proving the formula at the regularized
level, and then taking the limit (see also \cite[Lemma~2.4]{SP} for a similar
argument). 
Hence, based on the above considerations, the procedure implies
\begin{equation}\label{giu:a1a}
  \io \ln(1 + |\beta (\fhi\ee)|) | \beta(\fhi\ee) | \, \de x
   \le \io ( \mu\ee + \lambda \fhi\ee) \rho(\fhi\ee) \, \de x
   + \chi \io \sigma\ee \rho(\fhi\ee)\, \de x. 
\end{equation}
Then, in order to control the right-hand side, we first notice that 
\begin{equation}\label{giu:a2}
  \io ( \mu\ee + \lambda \fhi\ee) \rho(\fhi\ee) \, \de x
   \le \frac12 \| \mu\ee \|^2 + \frac{\lambda^2}2 \| \fhi\ee \|^2
   + \| \rho(\fhi\ee) \|^2.
\end{equation}
Moreover, it is clear that 
\begin{equation}\label{giu:a3}
  \| \rho(\fhi\ee) \|^2 
  = \io \ln^2(1 + |\rho(\fhi\ee)|) \, \de x
  \le \frac14   \io \ln(1 + |\beta (\fhi\ee)|) | \beta(\fhi\ee) | \, \de x
   + c,
\end{equation}
with $c$ independent of $\epsi$. The control of the last term in \eqref{giu:a1a} 
is a bit more delicate. To get it, we need to observe that $\psi(r)=e^r$, $r\in\RR$, and 
$\psi^*(s)= s(\ln s - 1)$, $s\ge 0$ (in fact $\psi^*$ may be intended to be 
identically $+\infty$ for $s<0$) are convex conjugate functions. Hence, we may apply the 
Fenchel-Young inequality, which gives (recall that $\sigma\ee\ge 0$ a.e.)
\begin{align}\nonumber
  \chi \io \sigma\ee \rho(\fhi\ee)\, \de x
   & \le \chi \io \psi^*(\sigma\ee)\, \de x + \chi \io \psi(|\rho(\fhi\ee)|) \, \de x \\
 \nonumber                                                                                                                                                                              
   & = \chi \io \sigma\ee (\ln \sigma\ee - 1) \, \de x 
    + \chi \io ( 1 + |\beta(\fhi\ee) | ) \, \de x\\
 \label{giu:a4}
  & \le \chi \io \sigma\ee \ln ( 1 + \sigma\ee) \, \de x 
    + \frac14   \io \ln(1 + |\beta (\fhi\ee)|) | \beta(\fhi\ee) | \, \de x 
    + c,
\end{align}
where the last inequality follows from elementary considerations,
similarly with \eqref{giu:a3}. 
Replacing \eqref{giu:a2}-\eqref{giu:a4} into \eqref{giu:a1a}, we then obtain 
\begin{equation}\label{giu:a1b}
  \frac12 \io \ln(1 + |\beta (\fhi\ee)|) | \beta(\fhi\ee) | \, \de x
   \le \frac12 \| \mu\ee \|^2 + \frac{\lambda^2}2 \| \fhi\ee \|^2
   + \chi \io \sigma\ee \ln ( 1 + \sigma\ee) \, \de x + c,
\end{equation}
whence, integrating in time and recalling \eqref{st:11}-\eqref{st:12} and \eqref{st:15},
we deduce the additional bound 
\begin{equation}\label{st:17}
  \| \beta(\fhi\ee) \ln(1 + |\beta (\fhi\ee)|) \|_{L^1(0,T;L^1(\Omega))} \le c.
\end{equation}
Next, in order to derive an entropy estimate (which, in turn, helps us to write
a weak formulation of \eqref{eq:2approx} 
suitable for taking the limit $\epsi\searrow 0$),
we multiply~\eqref{eq:2approx} by $ \frac{1}{\rs +\delta}$ for $\delta>0$. 
Actually, as $\delta>0$ is fixed, the function $(\cdot+\delta)^{-1}$ is bounded and 
globally Lipschitz over $[0,+\infty)$; hence, recalling that $\sigma\ee\ge 0$
almost everywhere, this procedure is admissible. It is then easy to infer
\begin{align}\nonumber
  & \ddt \ln (\sigma\ee + \delta) 
   - \Delta \ln (\sigma\ee + \delta)
   - | \nabla \ln (\sigma\ee + \delta) |^2
   - \chi \dive\bigg(\frac{\sigma\ee}{(\sigma\ee + \delta)(1+\epsi\sigma\ee)}\nabla(1-\fhi\ee) \bigg)\\
 \label{giu:a5}
  & \mbox{}~~~~~ - \chi \frac{\sigma\ee}{(\sigma\ee + \delta)^2(1+\epsi\sigma\ee)}
   \nabla\sigma\ee \cdot \nabla (1 - \fhi\ee)
  = \frac{\alpha(\fhi\ee,\sigma\ee)\sigma\ee }{\sigma\ee+\delta}.
\end{align}
Note that the chain rule formulas 
used to derive the above relation are apparently
formal, but could easily be made rigorous. Actually, as observed in Remark~\ref{rem:rego}, 
the regularity \eqref{rego:sigma1} is not optimal, and, possibly smoothing out the 
initial datum, better conditions could be proved at the 
level $\epsi\in(0,1)$. For instance, the chain rule formula used to deduce
$(\sigma\ee + \delta)^{-1} \partial_t \sigma\ee = \partial_t \ln (\sigma\ee + \delta)$
is rigorous as far as $\partial_t\sigma\ee$ lies, say, in $L^2$. This is not part of 
\eqref{rego:sigma1}, but, as said, it can be shown to hold if $\sigma_0^\epsi$ is 
smoother (we omit details for brevity).

To derive an additional estimate, we multiply \eqref{giu:a5} by $-1$
and integrate over $\Omega$. Using the no-flux boundary conditions 
\eqref{eq:3approx} and noticing that 
\begin{equation}\label{giu:a6}
  \chi \bigg| \io \frac{\sigma\ee \nabla\sigma\ee \cdot \nabla (1 - \fhi\ee)}%
  {(\sigma\ee + \delta)^2(1+\epsi\sigma\ee)} \bigg|
   \le \frac14 \| \nabla \ln (\sigma\ee + \delta) \|^2 
    + c \| \nabla \fhi\ee \|^2,
\end{equation}
with $c$ being independent both of $\epsi$ and of $\delta$, noting also that
the control of the right-hand side of \eqref{giu:a5} is straightforward
thanks to the uniform boundedness of $\alpha$, we 
readily obtain the bounds
\begin{equation}\label{st:18}
  \| \ln ( \sigma\ee + \delta ) \|_{L^\infty(0,T;L^1(\Omega))} 
   + \| \ln ( \sigma\ee + \delta ) \|_{L^2(0,T;V)} \le c.
\end{equation}
Next, we consider a test function $\theta\in W^{1,p}(\Omega)$, where $p>d$ so that 
we can take advantage of the continuous embedding 
$W^{1,p}(\Omega) \subset C^0({\overline \Omega})$,
and multiply \eqref{giu:a5} by $\theta$. Then, simple considerations
permit us to deduce that
\begin{equation}\label{st:19}
  \| \partial_t \ln ( \sigma\ee + \delta ) \|_{L^1(0,T;W^{1,p}(\Omega)')} \le c,
\end{equation}
where we point out once more that the constants $c$ occurring, in particular, in
\eqref{st:18}-\eqref{st:19} are independent of $\epsi$ and $\delta$. 
Finally, in order to derive the last \textit{a priori} estimate,
we may multiply~\eqref{eq:2approx} by $\varepsilon^2 \rs$ and 
integrate over $\Omega$ in order to observe 
\begin{align}\nonumber
    \frac{\de}{\de t}\frac{\varepsilon^2}{2}\| \rs\|^2 
      + \varepsilon^2\|\nabla \rs\|^2 =
      &  \int_{\Omega} \chi \frac{\varepsilon\rs}{1+\varepsilon\rs}\nabla \rp  \cdot \varepsilon\nabla \rs 
            + \alpha(\rp,\rs) \varepsilon^2|\rs|^2 \, \de x \\
 \label{st:19b}           
    \leq & \frac{\chi^2}{2}\| \rp\|_{V}^2 
     + \frac{\varepsilon^2}{2}\|\nabla \rs\|^2 
     + \ov{\alpha} \varepsilon^2 \|\rs\|^2.
\end{align}
Hence, using also~\eqref{sigmaepsi1}, 
Gr\"onwall's lemma allows us to deduce the bound 
\begin{equation}\label{apri:eps}
    \varepsilon \| \rs \|_{L^2(0,T;V)}\leq c \,. 
\end{equation}
Now, we aim to show that we can let $\delta\searrow 0$ in \eqref{giu:a5}.
In particular, we claim that 
\begin{equation}\label{st:e1}
  \ln (\sigma\ee + \delta) \to \ln \sigma\ee 
   \quad\text{strongly in }\,L^2(0,T;V).
\end{equation}
To prove this fact, we notice that, for a.e.~$t\in(0,T)$, a.e.\
on the set $\Omega_n=\Omega_n(t):=\{\sigma\ee(t) >1/n\}$, one has 
\begin{equation}\label{st:e2}
  | \nabla \ln (\rs+\delta)|^2 
  = \frac{|\nabla \rs|^2}{(\rs+\delta)^2} 
  \le \frac{|\nabla \rs|^2}{\rs^2}
  = |\nabla \ln \rs |^2.
\end{equation}
Indeed, as far as its argument stays larger than $n^{-1}$, the function 
$\ln$ is Lipschitz; hence the above follows from
standard chain rule formulas in Sobolev spaces.
On the other hand, as $\sigma\ee (\cdot,t) > 0$ almost everywhere in $\Omega$
(and for all $t\in[0,T]$ with a possible exception of a set of zero
measure), \eqref{st:e2} holds on the union, for $n\in\NN$, of the sets $\Omega_n$,
which coincides with $\Omega$ up to a set of zero measure. Hence, 
integrating \eqref{st:e2} in space and time, one readily deduces
\begin{equation}\label{st:e3}
\begin{aligned}
    \|\nabla \ln \rs \|_{L^2(0,T;V)}^2 &\le \liminf_{\delta\searrow 0} \| \nabla \ln (\rs+\delta)\|_{L^2(0,T;V)}^2\\
    & \leq  \limsup_{\delta\searrow 0} \| \nabla \ln (\rs+\delta)\|_{L^2(0,T;V)}^2
   \le \|\nabla \ln \rs \|_{L^2(0,T;V)}^2,
\end{aligned}
\end{equation}
where the first inequality follows from the point-wise a.e.~convergence and Fatou's lemma. This 
gives the desired \eqref{st:e1}.
As a consequence, it is easy to realize that
one can take the limit $\delta\searrow 0$ in \eqref{giu:a5} so to 
deduce 
\begin{align}\nonumber
  & \ddt \ln \sigma\ee 
   - \Delta \ln \sigma\ee
   - | \nabla \ln \sigma\ee |^2
   - \chi \dive\Big(\frac{1}{1+\epsi\sigma\ee}\nabla(1-\fhi\ee) \Big)\\
 \label{giu:a5b}
  & \mbox{}~~~~~ - \chi \frac{1}{1+\epsi\sigma\ee}
   \nabla \ln\sigma\ee \cdot \nabla (1 - \fhi\ee)
  = \alpha(\fhi\ee,\sigma\ee).
\end{align}
Note that, as a byproduct of the argument, we have improved the information 
\eqref{posi} on the sign properties of $\sigma\ee$, which turns out to be 
{\sl strictly}\/ positive almost everywhere. 


\subsection{Proof of Theorem~\ref{thm:global}: passage to the limit}
\label{subsec:lim}

Our next task consists in taking the limit $\epsi\searrow 0$, and, to
this purpose, we start considering relation \eqref{eq:2approx}. First of all, we observe
that from \eqref{st:11}, \eqref{st:15}-\eqref{st:16}, and the Aubin--Lions lemma, there follows that 
\begin{align}\label{co:11}
  & \fhi\ee \to \fhi \quad\text{weakly star in }\, H^1(0,T;V') \cap L^\infty(0,T;V)
   ~~\text{and strongly in }\,L^\infty(0,T;H),\\
 \label{co:12}
  & \mu\ee \to \mu \quad\text{weakly in } L^2(0,T;V).
\end{align}
Moreover, noting that the analogue of \eqref{st:18}-\eqref{st:19} holds
for $\delta=0$ thanks to semicontinuity of norms with respect to weak or
weak star convergence, one also has
\begin{align}\label{co:13}
  & \ln\sigma\ee \to \ell \quad\text{weakly star in }\, BV(0,T;W^{1,p}(\Omega)') 
    \cap L^2(0,T;V),~~\text{strongly in }\,L^2(0,T;H),
\end{align}
for a suitable limit function $\ell$. Here and below, it is intended that all 
convergence relations hold for a suitable (nonrelabelled) subsequence
of $\epsi\searrow 0$. Note also that we applied above a generalized version of the 
Aubin--Lions lemma (see, e.g., \cite[Corollary~7.9]{RoubicekBook} or \cite{SimonComp}). 

As in particular a.e.~convergence holds
in $Q$, the equintegrability condition \eqref{st:12} readily implies
\begin{align}\label{co:14}
  & \sigma\ee = e^{\ln \sigma\ee} \to e^\ell =: \sigma
   \quad 
   \text{strongly in }\,L^P(0,T;L^1(\Omega)).
\end{align}
Here and below, $P$ will be a generic exponent such that $P\in[1,\infty)$. 
In the sequel, we will mostly privilege the notation $\sigma$ and
write $\ln \sigma$ in place of $\ell$. From \eqref{co:13}
we obtain in particular that $\sigma>0$ a.e.~in~$Q$. 
Now, \eqref{co:14} also implies that 
\begin{align}\label{co:14b}
  \sigma\ee^{1/2} \to \sigma^{1/2} \quad\text{weakly star in }\,L^\infty(0,T;H)
  ~~\text{and strongly in }\,L^P(0,T;H).
\end{align}
Similarly, we claim that 
\begin{equation}\label{co:14c}
  \frac{ \sigma\ee^{1/2} }{ (1+\epsi\rs)^{1/2} } 
     \to \sigma^{1/2} \quad\text{weakly star in }\,L^\infty(0,T;H)
  ~~\text{and strongly in }\,L^P(0,T;H).
\end{equation}
To prove the above, one may notice that 
\begin{align*}
  \left| \frac{ \sigma\ee^{1/2} }{ (1+\epsi\rs)^{1/2} } - \sigma^{1/2} \right|
  & \le \frac{ 1 }{ (1+\epsi\rs)^{1/2} } \big| \rs^{1/2} - \sigma^{1/2} \big|
   + | \sigma^{1/2} | \frac { | (1+\epsi\rs)^{1/2} - 1 |} {(1+\epsi\rs)^{1/2}}\\
  & \le \big| \rs^{1/2} - \sigma^{1/2} \big|
   + | \sigma^{1/2} | \frac { | (1+\epsi\rs)^{1/2} - 1 |} {(1+\epsi\rs)^{1/2}}
\end{align*}
and observe that, since $\sigma$ is finite almost everywhere, the last
factor tends to $0$ almost everywhere; hence, \eqref{co:14c} follows
from Lebesgue's dominated convergence theorem.
 
Let us now take the limit of the cross-diffusion term in \eqref{eq:2approx}.
First of all, we observe that, from \eqref{st:13}, there follows
\begin{equation}\label{co:15}
  \frac{\rs^{1/2}}{(1+\varepsilon\rs)^{1/2}} 
     \nabla ( \ln \rs + \varepsilon(\rs -1)+ \chi (1-\rp) )
  \to \eta \quad\text{weakly in }\,L^2(0,T;H).
\end{equation}
In order to identify the function $\eta$, we may notice that 
the bound~\eqref{apri:eps} and with the second of \eqref{st:18} (which still holds
for $\delta=0$) imply that there exists $\zeta \in L^2(0,T;H)$ such that
\begin{equation}
    \nabla( \ln  \rs + \varepsilon (\rs-1)) \to  \zeta 
   \quad\text{weakly in }\,L^2(0,T;H),
\end{equation}
which can be identified to be $\nabla \ln \sigma$ due to~\eqref{co:13} and~\eqref{st:11new}. 
Relation~\eqref{co:14c} now guarantees the convergence
\begin{equation}\label{co:16}
  \frac{\rs^{1/2}}{(1+\varepsilon\rs)^{1/2}} 
     \nabla ( \ln \rs + \varepsilon(\rs -1) )
  \to \sigma^{1/2} \nabla \ln \sigma
   \quad\text{weakly in }\,L^1(0,T;L^1(\Omega)).
\end{equation}
On the other hand, \eqref{co:11} and \eqref{co:14c} imply 
\begin{equation}\label{co:15x}
  \frac{\rs^{1/2}}{(1+\varepsilon\rs)^{1/2}} 
     \nabla ( \chi (1-\rp) )
  \to \sigma^{1/2} \nabla ( \chi (1-\fhi) )
  \quad\text{weakly in }\,L^P(0,T;L^1(\Omega)).
\end{equation}
Combining \eqref{co:16} and \eqref{co:15x}, we then conclude that 
\begin{equation}\label{id:eta}
  \eta = \sigma^{1/2} \nabla ( \ln \sigma + \chi (1-\fhi) ).
\end{equation}
Let us now consider relation \eqref{sig:weak} and notice that,
combining \eqref{co:14c} and \eqref{co:15}, using also 
\eqref{id:eta}, there follows
\begin{align}\nonumber
  & \Big( \frac{\sigma\ee}{1+\epsi\sigma\ee} \Big)
  \nabla ( \ln \sigma\ee + \epsi (\sigma\ee - 1) + \chi ( 1-\rp) ) \\
 \nonumber
  & \qquad = \Big( \frac{\sigma\ee}{1+\epsi\sigma\ee} \Big)^{1/2}
   \Big( \frac{\sigma\ee}{1+\epsi\sigma\ee} \Big)^{1/2}
      \nabla ( \ln \sigma\ee + \epsi (\sigma\ee - 1) + \chi ( 1-\rp) ) \\
 \label{co:17}
  & \qquad \to \sigma^{1/2} \eta = \sigma \nabla ( \ln \sigma + \chi (1-\fhi) )
   \quad\text{weakly in }\,L^2(0,T;L^1(\Omega)).
\end{align}
Next, using relations \eqref{co:11}, \eqref{co:14} and Hypothesis~\ref{hypo:weak}, 
it is easy to deduce
\begin{equation}\label{co:alpha}
  \alpha(\fhi\ee,\sigma\ee) \to \alpha(\fhi,\sigma)
   \quad\text{weakly star in }\,L^\infty(Q)~~
   \text{and strongly in }\,L^P(Q).
\end{equation}
This fact, combined with \eqref{co:14}, allows for taking the limit of the right-hand 
side of \eqref{sig:weak}.
Moreover, testing \eqref{sig:weak} by $\vartheta\in W^{2,p}(\Omega)$ 
for $p>d$
and using the continuous embedding $W^{2,p}(\Omega)\subset C^1({\ov \Omega})$,
we may deduce 
\begin{equation}\label{co:18}
  \partial_t \rs \to \partial_t \sigma
   \quad\text{weakly in }\,L^2(0,T;W^{2,p}(\Omega)').
\end{equation}
The above considerations suffice in order to take the limit $\epsi\searrow 0$
in \eqref{sig:weak}, interpreted as a relation in $L^2(0,T;W^{2,p}(\Omega)')$
(or, more precisely, in its weak formulation obtained by means of the function
$\vartheta$) and to get back \eqref{weak:sig}.

\smallskip

Next, we take the limit $\epsi\searrow 0$ in the Cahn-Hilliard system
\eqref{eq:1approx}. Actually, the procedure is mostly standard, with the 
only remarkable point being related with the nonlinear term depending
on $F'$. To deal with it, it is sufficient to observe that, thanks to
the pointwise convergence resulting from \eqref{co:11} and to the 
equiintegrability condition \eqref{st:17}, one has
\begin{equation}\label{co:21}
  \beta(\fhi\ee) \to \beta(\fhi) 
   \quad\text{strongly in }\,L^1(0,T;L^1(\Omega)),
\end{equation}
and an analogous property holds also for $F'(\fhi\ee)$, which differs from
$\beta(\fhi\ee)$ just by a linear function of $\fhi\ee$. 
Then, it is also worth observing that the conditions in \eqref{reg:Fprime} 
are a consequence of \eqref{st:15}, \eqref{st:17},~\eqref{co:11} and Fatou's lemma. 

Next, let us take $\psi\in V \cap L^\infty(\Omega)$: testing the second of
\eqref{eq:1approx} by $\psi$ and performing standard integrations 
by parts, we deduce the weak formulation
\begin{equation}\label{eq:1psi}
  \io \mu\ee \psi \,\de x
   = \io \nabla\fhi\ee\cdot \nabla \psi 
   + ( \beta(\fhi\ee) - \lambda \fhi\ee - \chi \sigma\ee) \psi \,\de x.
\end{equation}
Then, writing \eqref{eq:1psi} for $\epsi_1,\epsi_2 \in (0,1)$, 
taking the difference, and choosing $\psi = \fhi_{\epsi_1} - \fhi_{\epsi_2}$,
integrating in time, using \eqref{co:14} and the property
\begin{equation}\label{co:infty}
  \fhi\ee \to \fhi 
   \quad\text{weakly star in }\,L^\infty(0,T;L^\infty(\Omega)),
\end{equation}
which is a consequence of \eqref{trepunti},
it is not difficult to arrive at the Cauchy estimate leading to
\begin{equation}\label{co:22}
  \fhi\ee \to \fhi
   \quad\text{strongly in }\,L^2(0,T;V),
\end{equation}
which complements \eqref{co:11}. Hence, we can take the limit 
in \eqref{eq:1psi} and get the weak formulation
of \eqref{eq:1approx}. Moreover, the previous estimates and a
comparison of terms in the second \eqref{eq:1approx} 
(or, more precisely, the properties of the modular in the Orlicz
space $L\log L$) imply that
\begin{equation}\label{co:23}
  \int_0^T \int_\Omega | \Delta\fhi\ee | \ln (1 + | \Delta\fhi\ee | )
   \, \de x \, \de t \le c.
\end{equation}
Then, using Dunford-Pettis' theorem, we deduce \eqref{reg:Dphi} and the 
validity of \eqref{weak:mu} as a point-wise relation as well. 
In turn, the trace theorem~\cite[Prop.~3.80]{Demengel}
permits us to interpret also~\eqref{bou:fhi} 
as a point-wise a.e.-relation. 

Finally, we consider the entropy relation \eqref{giu:a5b}, which holds
as an equality at the level $\epsi>0$, but will turn to an inequality in the 
limit due to the quadratic term on the left-hand side. First of all,
we prove a reinforced estimate on $\ell = \ln \sigma$, which was not
observed in \cite{GiulioNew} and relies in an essential way 
on the second condition in \eqref{hp:sigma00}. 

To this aim, we set 
\begin{equation}\label{defi:gamma}
   \gamma(r):= - \ln ( 1 + r_- ),
\end{equation}
where $r_-$ denotes 
the negative part of $r$, and notice that the function
$\gamma(r)$ is identically $0$ for $r\ge 0$ and is strictly negative
for $r<0$; moreover, it is globally Lipschitz continuous and monotone 
in the whole of $\RR$. We also define 
\begin{equation}\label{gammaciapo}
  \gammaciapo(r):= - \int_r^0 \gamma(s)\, \de s
   = (1 + r_-) \ln (1 + r_-) - r_-
\end{equation}
and observe that, using $\gammaciapo$, condition \eqref{hp:sigma00} can
be equivalently rewritten as $\gammaciapo(\ln\sigma_0) \in L^1(\Omega)$. Then, 
we test \eqref{giu:a5b} by $\gamma(\ell\ee)=\gamma(\ln\sigma\ee)$: we just remark
that this formal procedure could by made rigorous by truncating $\gamma$ and/or
working at the level $\delta>0$ (i.e., considering \eqref{giu:a5}); we omit
details for brevity. That said, the resulting relation
(where we have neglected a nonnegative term from the left-hand side)
reads
\begin{align}\nonumber
  & \ddt \io \gammaciapo (\ell\ee) \, \de x
   + \io | \gamma(\ell\ee) | | \nabla \ell\ee |^2 \, \de x
  \le \chi \io \frac{1}{1+\epsi\sigma\ee} 
      \gamma'(\ell\ee)\nabla\fhi\ee \cdot \nabla\ell\ee\, \de x\\
 \label{co:24}
  & \qquad\quad
     - \chi \io \frac{1}{1+\epsi\sigma\ee} 
       \gamma(\ell\ee) \nabla\fhi\ee \cdot \nabla\ell\ee \, \de x
     + \io \alpha(\fhi\ee,\sigma\ee) \gamma(\ell\ee) \, \de x
   =: J_1 - J_2 + J_3.
\end{align}
Now, by the boundedness of $\alpha$, the global Lipschitz continuity of $\gamma$, 
and the property $\gamma(0)=0$, we have
\begin{equation}\label{co:24b}
  J_1 + J_3 \le c \| \nabla \fhi\ee \| \| \nabla \ell\ee \|
   + c \| \ell\ee \|_{L^1(\Omega)}.
\end{equation}
The control of $J_2$ is a bit more delicate. 
First of all, by Young's inequality, we have
\begin{equation}\label{co:25}
  |J_2| \le c \io | \gamma(\ell\ee) | | \nabla \fhi\ee | | \nabla \ell\ee |\, \de x
   \le \frac12 \io | \gamma(\ell\ee) | | \nabla \ell\ee |^2 \, \de x
   - c \io \gamma(\ell\ee) | \nabla \fhi\ee |^2 \, \de x
\end{equation}
(recall that $\gamma$ is nonpositive), and the last term can be
integrated by parts and managed via the Fenchel-Young inequality
(where $\psi$ and $\psi^*$ are as in \eqref{giu:a4})
as follows:
\begin{align}\nonumber
  & - c \io \gamma(\ell\ee) | \nabla \fhi\ee |^2 \, \de x
  = c \io \gamma'(\ell\ee) \fhi\ee \nabla\ell\ee \cdot \nabla\fhi\ee\, \de x
   + c \io \gamma(\ell\ee) \fhi\ee \Delta\fhi\ee\, \de x\\
 \nonumber
  & \qquad\quad \le c \| \nabla\ell\ee \| \| \nabla\fhi\ee \|
   + c \io \psi(|\gamma(\ell\ee)|)\, \de x + c \io \psi^*(|\Delta\fhi\ee|) \, \de x\\
 \label{co:26}  
  & \qquad\quad \le c \| \nabla\ell\ee \| \| \nabla\fhi\ee \|
    + c + c \| (\ell\ee)_- \|_{L^1(\Omega)}
    + c \int_\Omega | \Delta\fhi\ee | \ln (1 + | \Delta\fhi\ee | ) \, \de x.
\end{align}
In the above we have also used the uniform boundedness of $\fhi\ee$ 
and of $\gamma'$ (or, more precisely, the global Lipschitz
continuity of $\gamma$). Then, using \eqref{co:24b}-\eqref{co:26},
\eqref{co:24} gives
\begin{align}\nonumber
  & \ddt \io \gammaciapo (\ell\ee) \, \de x
   + \frac12 \io | \gamma(\ell\ee) | | \nabla \ell\ee |^2 \, \de x\\
 \label{co:24y}   
  & \qquad\quad \le c \Big( 1 + \| \ell\ee \|_{L^1(\Omega)}
   + \| \nabla \fhi\ee \| \| \nabla \ell\ee \| 
   + \int_\Omega | \Delta\fhi\ee | \ln (1 + | \Delta\fhi\ee | ) \, \de x \Big).
\end{align}
Moreover, by means of \eqref{co:11}, \eqref{co:13}, and \eqref{co:23},
integrating in time, and using \eqref{sigmaepsi01}   together
with assumption \eqref{hp:sigma00},  we deduce
\begin{equation}\label{st:31}
  \| \gammaciapo(\ell\ee) \|_{L^\infty(0,T;L^1(\Omega))} \le c.
\end{equation}
Now, by the trivial inequality 
\begin{equation}\label{st:31b}
  |r| \ln ( 1 + |r|) \le r_+^2 + r_- \ln ( 1 + r_- ), 
\end{equation}
the above implies in particular
\begin{equation}\label{st:31c}
  \big\| |\ell\ee| \ln (1 + |\ell\ee|) \big\|_{L^\infty(0,T;L^1(\Omega))} \le c,
\end{equation}
whence, using the pointwise convergence, we deduce 
\begin{align}\label{co:14d}
  & \ln \sigma\ee \to \ln\sigma 
   \quad\text{weakly star in }\, L^\infty(0,T;L^1(\Omega))
  ~~\text{and strongly in }\,L^P(0,T;L^1(\Omega)),
\end{align}
for every $P\in [1,\infty)$, which completes condition \eqref{red:ln}.

\bigskip


\subsection{Proof of Theorem~\ref{thm:global}: entropy and energy inequalities}
\label{subsec:enen}

To conclude the proof of Theorem~\ref{thm:global}, we prove the entropy 
and energy inequalities complementing the concept of weak solution.
To this aim, we first observe that, from \eqref{st:11new}, \eqref{co:13} and elementary 
considerations, there follows that 
\begin{align}\label{co:14e}
  & \frac1{1+\epsi\sigma\ee} \nabla \ln \sigma\ee \to \nabla \ln\sigma 
    \quad\text{weakly in }\,L^2(0,T;H).
\end{align}
Hence, recalling \eqref{co:22}, one has, at least,
\begin{equation}\label{co:14f}
  \frac{1}{1+\epsi\sigma\ee}
   \nabla \ln\sigma\ee \cdot \nabla (1 - \fhi\ee)
    \to \nabla \ln\sigma \cdot \nabla (1 - \fhi)
    \quad\text{weakly in }\,L^1(0,T;L^1(\Omega)).
\end{equation}
Let us now go back to the regularized version \eqref{giu:a5b} of the entropy inequality.
Then, in order to convert it into an integral relation, we multiply it by (minus) a test 
function $\theta\in C^1([0,T];W^{1,p}(\Omega))$, $p>d$. 
Integrating by parts some terms and performing standard manipulations, setting also 
$\ell\ee = \ln \sigma\ee$ for simplicity, we deduce
\begin{align}\nonumber
  & \int_0^T \io \ell\ee \partial_t \theta \, \de x \, \de t
   - \io \ell\ee(T) \theta(T) \, \de x
   - \int_0^T \io \nabla \ell\ee \cdot \nabla \theta \, \de x \, \de t
   + \int_0^T \io | \nabla \ell\ee |^2 \theta_+ \, \de x \, \de t\\
 \nonumber  
  & \mbox{} ~~~~~
    - \chi \int_0^T \io \frac{1}{1+\epsi \sigma\ee} 
       \big( \nabla (1 - \fhi\ee) \cdot \nabla \theta
       - \nabla \ell\ee \cdot \nabla (1 - \fhi\ee) \theta \big) \, \de x \, \de t\\
 \label{entr:11}
  & \mbox{} ~~~~~
   = - \int_0^T \io \alpha(\fhi\ee,\sigma\ee) \theta \, \de x \, \de t
    - \io \ell\ee(0) \theta(0) \, \de x
    + \int_0^T \io | \nabla \ell\ee |^2 \theta_- \, \de x \, \de t.
\end{align}
In principle, we would like to take the $\limsup$, as $\epsi\searrow0$, 
in the above relation. However, in order to manage the last term, we first 
need to add and simultaneously subtract from the right-hand side the quantity
\begin{equation}\label{sem:10} 
  \| \theta_- \|_{L^\infty(\Omega\times (0,T))} \int_0^T \io | \nabla \ell\ee|^2 \, \de x \, \de t.
\end{equation}
The presence of such a correction term is what actually permits us to take the 
$\limsup$. In particular, we may observe that, thanks to \eqref{co:13}, \eqref{co:14d},
\eqref{co:14e}, \eqref{co:22} and \eqref{co:alpha}, for all the terms on the left-hand side
of \eqref{entr:11}, there exists the $\lim_{\epsi\searrow 0}$ (and not just 
the $\limsup$), and it takes the expected value. The only exception is represented 
by the fourth term, for which, however, we may observe that 
\begin{equation}\label{sem:11} 
  \int_0^T \io | \nabla \ell |^2 \theta_+ \, \de x \, \de t
   \le \liminf_{\epsi\searrow0} \int_0^T \io | \nabla \ell\ee |^2 \theta_+ \, \de x \, \de t
   \le \limsup_{\epsi\searrow0} \int_0^T \io | \nabla \ell\ee |^2 \theta_+ \, \de x \, \de t,
\end{equation}
by weakly lower semicontinuity. Some more words need to be spent on 
the relation
\begin{equation}\label{sem:12}
  \lim_{\epsi\searrow0}  - \io \ell\ee(T) \theta(T) \, \de x
   = - \langle \ell(T), \theta(T) \rangle.
\end{equation}
To obtain it, we first observe that, as $W^{1,p}(\Omega)$ for $p\in(d,\infty)$
is reflexive and separable, with separable dual,
by a suitable form of Helly's selection principle
(see, e.g., \cite[Thm.~3.1]{MaMi2009}),
we may complement \eqref{st:31b} with
\begin{equation}\label{hel}
  \ell\ee(t) \to \ell(t) \quad\text{weakly star in }\,W^{1,p}(\Omega)'
   ~~\text{for {\bf every} }\, t\in[0,T],
\end{equation}
which in particular entails \eqref{sem:12}. 

\smallskip

Moving to the terms on the right-hand side of \eqref{entr:11}, 
it is apparent that the first two integrals converge to the expected limits
(in particular, for the integral depending on the initial datum, the dominated convergence
theorem is used). Hence, we only need to take care of the contributions quadratically 
depending on $\nabla\ell\ee$. Namely, recalling that we added and subtracted the 
term in \eqref{sem:10}, we need to manage the quantities
\begin{equation*}
  \int_0^T \io | \nabla \ell\ee |^2 \big( \theta_- - \| \theta_- \|_{L^\infty(\Omega\times (0,T))} \big) \, \de x \, \de t
   + \| \theta_- \|_{L^\infty(\Omega\times (0,T))} \int_0^T \io | \nabla \ell\ee|^2 \, \de x \, \de t.
\end{equation*}
As for the first integral, we clearly have 
\begin{align}\nonumber
  & \limsup_{\epsi\searrow 0} \int_0^T \io | \nabla \ell\ee |^2 
      \big( \theta_- - \| \theta_- \|_{L^\infty(\Omega\times (0,T))} \big) \, \de x \, \de t\\
 \nonumber
  & \mbox{}~~~~~
    = - \liminf_{\epsi\searrow 0} \int_0^T \io | \nabla \ell\ee |^2 
       \big( \| \theta_- \|_{L^\infty(\Omega\times (0,T))} - \theta_- \big) \, \de x \, \de t\\
 \label{sem:13}
  & \mbox{}~~~~~\le \int_0^T \io | \nabla \ell |^2 
     \big( \theta_- - \| \theta_- \|_{L^\infty(\Omega\times (0,T))} \big) \, \de x \, \de t.
\end{align}
Concerning the latter quantity in \eqref{sem:13}, we simply notice that, by
estimate \eqref{st:18} (or, more precisely, its analogue for $\delta=0$), 
we have
\begin{equation}\label{sem:14}
 \| \theta_- \|_{L^\infty(\Omega\times (0,T))} \int_0^T \io | \nabla \ell\ee|^2 \, \de x \, \de t
  \le \| \theta_- \|_{L^\infty(\Omega\times (0,T))} \zeta_\varepsilon,
\end{equation}
with $\zeta_\varepsilon :=\int_0^T \io | \nabla \ell\ee|^2 \, \de x \, \de t$ 
being a bounded sequence of real numbers such that 
$\zeta_\varepsilon \to \zeta \geq 0 $ for a subsequence. In particular, the limit 
$\zeta$  is bounded by a computable quantity depending only on the initial datum 
and on the fixed physical parameters. 

Collecting all the above considerations, we then end up with
\begin{align}\nonumber
  & \int_0^T \io \ell \partial_t \theta \, \de x \, \de t
   - \langle \ell(T) , \theta(T) \rangle
  + \int_0^T \io \theta | \nabla \ell |^2 \, \de x \, \de t
   - \int_0^T \io \nabla \ell \cdot \nabla \theta \, \de x \, \de t\\
 \nonumber  
  & \mbox{} ~~~~~
    - \chi \int_0^T \io \big( \nabla (1 - \fhi) \cdot \nabla \theta
       - \nabla \ell \cdot \nabla (1 - \fhi) \theta \big) \, \de x \, \de t\\
 \label{entr:13}
  & \mbox{}
   \le - \int_0^T \io \alpha(\fhi,\sigma) \theta \, \de x \, \de t
    - \io \ell_0 \theta(0) \, \de x
    + \| \theta_- \|_{L^\infty(\Omega\times (0,T))}
       \Big( \zeta - \int_0^T \io | \nabla \ell |^2  \, \de x \, \de s \Big).
\end{align}
which, after rearranging, reduces exactly to \eqref{weakentropy}. 
We finally point out that, as observed in Remark~\ref{rem:weakstrong},
using in particular \eqref{hel}, one can adapt the above argument 
in order to obtain the analogue of \eqref{weakentropy} on any subinterval $(0,t)$.

\bigskip

Next, we move to the energy inequality, which, rewritten in the integral form, 
will also account for correction terms in the limit $\epsi\searrow 0$, as 
expressed in the statement of Theorem~\ref{thm:global}. To obtain it, 
it is convenient to start by rewriting relation \eqref{enregest}
in the more ``compact'' form
\begin{equation}\label{enre1}
  \ddt \calF\ee + \calD\ee \le 
   \int_\Omega \alpha(\rp,\rs)\rs \big( \ln \rs + \varepsilon (\rs-1)
       +\chi (1-\rp)\big)\,\de x =: \calR\ee.
\end{equation}
Here, we have set $\calF\ee : = \calE\ee + K$ (recall that the approximate energy $\calE\ee$
was defined in \eqref{defi:Eee}), where the constant $K>0$ is chosen such that 
$\calF\ee \ge 0$ (note also that $K$ may be taken independently of $\epsi$); 
moreover, the ``dissipation term'' is defined as
\begin{equation}\label{enre2}
  \calD\ee 
   = \io\frac{\rs}{1+\varepsilon\rs} \big| \nabla ( \ln \rs + \varepsilon(\rs -1)+ \chi (1-\rp)) \big|^2
     + | \nabla \rmu|^2 \,\de x.
\end{equation}
Then, recalling \eqref{rhs:11}, we see that 
\begin{equation}\label{enre2.5}  
  |\calR\ee| \le c (1 + \calF\ee),
\end{equation}
whence, applying Gr\"onwall's lemma, we may deduce
\begin{equation}\label{enre3}
  \calF\ee(t) \le M_0(T) \quad\text{for every }\,t\in[0,T],
\end{equation}
where the computable constant $M_0>0$ depends on the initial energy and (exponentially)
on $T$, but it is independent of $\epsi$.

As before, we will consider the {\sl integral version}\/ of inequality \eqref{enre1}
and will get a a correction term in the limit. This is due to the fact that, on the one hand, for the
term $\sigma\ee \ln \sigma\ee$ appearing in $\calR\ee$ no strong $L^p$- convergence is available,
and, on the other hand, the function $\alpha$ has no sign properties; hence, the corresponding
contribution cannot be managed by means of semicontinuity methods, at least in a direct way.

To detail our argument, we first refine the control of $\calR\ee$: noting 
as $r\ee$ the integrand in $\calR\ee$, namely $\calR\ee = \io r\ee \, \de x$, we can
decompose it as $r\ee = r_{\epsi,1} + r_{\epsi,2}$, where
\begin{align}\label{ree12}  
  r_{\epsi,1} &: = \alpha(\fhi\ee,\sigma\ee) \big( \sigma\ee (\ln\sigma\ee - 1) + \epsi (\sigma\ee - 1)^2 + 1 \big),\\
  \label{ree22}  
  r_{\epsi,2} &: = \alpha(\fhi\ee,\sigma\ee) \big( \sigma\ee + \epsi (\sigma\ee - 1) 
    + \chi \sigma\ee ( 1 - \fhi\ee) - 1 \big).
\end{align}
Integrating \eqref{enre1} over a generic time interval $(s,t)$ with $0\le s < t \le T$,
using the above decomposition, we then have
\begin{equation}\label{enre4}
  \calF\ee(t) \le \calF\ee(s) 
   - \int_s^t \calD\ee(\tau) \, \de \tau
    + \int_s^t \int_\Omega r_{\epsi,1}(x,\tau) \, \de x \, \de \tau
     + \int_s^t \int_\Omega r_{\epsi,2}(x,\tau) \, \de x \, \de \tau.
\end{equation}
Now, we would like to take the $\liminf$, as $\epsi\searrow 0$, of the above relation. 
Then, setting 
\begin{equation}\label{defiD}
  \calD := \io \sigma \big| \nabla ( \ln \sigma  + \chi (1-\fhi)) \big|^2
     + | \nabla \mu|^2 \,\de x,
\end{equation}
we claim that the following relations hold:
\begin{align}\label{semi11}  
  \calF(t) := \calE(t) + K & \le \liminf_{\epsi\searrow 0} \calF\ee(t),\\
 \label{semi12}
  \limsup_{\epsi\searrow 0} \Big( - \int_s^t \calD\ee(\tau) \, \de \tau \Big)
   & = - \liminf_{\epsi\searrow 0} \int_s^t \calD\ee(\tau) \, \de \tau \
   \le - \int_s^t \calD(\tau) \, \de \tau,
\end{align}
Indeed, using \eqref{co:14}, \eqref{co:infty} and \eqref{co:22}, \eqref{semi11} follows 
directly from Fatou's lemma, whereas \eqref{semi12} is a consequence of \eqref{co:12}, \eqref{co:15}
(where the limit quantity $\eta$ is identified thanks to~\eqref{id:eta}),
and lower semicontinuity of norms with respect to weak convergence.

It remains to manage the last two integral terms in \eqref{enre4}, which are 
more delicate to deal with. First of all, we notice that, from 
\eqref{enre1}, \eqref{enre2.5}-\eqref{enre3} and the fact $\calD\ee \ge 0$, 
there follows
\begin{equation}\label{RBV}
  \calF\ee(t+h) - \calF\ee(t) \le \int_t^{t+h} | \calR\ee(\tau) |\,\de \tau
   \le {\tilde c} h, 
\end{equation}
for every $t\in[0,T)$ and every (sufficiently small) $h>0$, where ${\tilde c}>0$ is, 
as usual, a computable constant depending only on the problem data.

From the above, we readily obtain $\calF\ee\in \Bv([0,T])$: to see
this it is sufficient to observe that \eqref{RBV} implies that the function
$g\ee := \calF\ee - {\tilde c} \Id$ is non-increasing. More precisely, 
noting that, for any finite partition 
$0=t_0 < t_1 < \dots < t_n =T$ of $[0,T]$, there holds
\begin{equation}\label{bound:bv}
  \sum_{k=1}^n | \calF\ee(t_k) - \calF\ee(t_{k-1}) | 
   \le | \calF\ee(T) - \calF\ee(0) | + 2 {\tilde c} T,
\end{equation}
with ${\tilde c}$ as in \eqref{RBV}, recalling also that $\calF\ee$ is 
bounded in $L^\infty(0,T)$ independently of $\epsi$ (cf.\ \eqref{enre3}),
we immediately deduce that the total variation of $\calF\ee$ is uniformly bounded 
with respect to $\epsi$. 

As a consequence, applying once more Helly's selection principle, we may assume that, 
for {\sl every}\ $s\in[0,T]$, there exists the $\lim_{\epsi\searrow 0} \calF\ee(s)$.
Next, using \eqref{co:alpha} together with \eqref{co:14} and \eqref{co:infty}, we
readily obtain
\begin{equation}\label{co:re2}
  r_{\epsi,2} \to \alpha(\fhi,\sigma) \big( \sigma + \chi\sigma( 1 - \fhi) - 1\big)=: r_2
   \quad\text{say, weakly in }\,L^1(Q).
\end{equation}
In order to deal with the term $r_{\epsi,1}$, we first observe that, if
${\ov \alpha}\le 0$ (and consequently $\alpha$ is nonpositive), then
we may directly conclude that
\begin{equation}\label{alphaneg}
   \limsup_{\epsi\searrow 0} \int_s^t\io r_{\epsi,1} \,\de x \, \de \tau
    \le \int_s^t\io \alpha(\fhi,\sigma) \big( \sigma(\ln \sigma - 1) + 1 \big) \,\de x \, \de \tau,
\end{equation}
by means of Fatou's lemma. Hence, in this case, we obtain the 
energy inequality with no need for correction terms.

  Conversely, let us 
assume that ${\ov \alpha}>0$. In this case, we may decompose 
$r_{\epsi,1}$ as follows:
\begin{align}\nonumber
  r_{\epsi,1} & = (\alpha(\fhi\ee,\sigma\ee) - {\ov\alpha} ) \big( \sigma\ee (\ln\sigma\ee - 1)
    + \epsi (\sigma\ee - 1)^2 + 1 \big)
    + {\ov \alpha} \big( \sigma\ee (\ln\sigma\ee - 1) + \epsi (\sigma\ee - 1)^2 + 1 \big)\\
 \label{dec:re1} 
  & =: r_{\epsi,11} + r_{\epsi,12}.
\end{align}
Noting that $r_{\epsi,11} \le 0$ almost everywhere and applying Fatou's lemma, we then deduce
\begin{align}\nonumber
  \limsup_{\epsi\searrow 0} \int_s^t\io r_{\epsi,11} \,\de x \, \de \tau
   & = - \liminf_{\epsi\searrow 0} \int_s^t\io (- r_{\epsi,11} ) \,\de x \, \de \tau\\
  \label{enre6} 
   & \le \int_s^t\io (\alpha(\fhi,\sigma) - {\ov\alpha} ) \big( \sigma (\ln\sigma - 1) + 1 \big) \,\de x \, \de \tau.
\end{align}
Now, defining $Z_\varepsilon:= \int_0^T \io r_{\epsi,12} \,\de x \, \de t$, 
we deduce that $ \{ Z_\varepsilon\}$ is a bounded sequence of real-valued numbers such that $ Z_\varepsilon \to Z\geq 0$ at least for a subsequence.
From the a-priori estimates \eqref{st:12} and \eqref{st:11new} it can be observed 
that these numbers are bounded by the assigned problem parameters and initial values. We then conclude that 
\begin{equation}\label{Zeta}
  \int_s^t \io r_{\epsi,12} \,\de x \, \de \tau 
   = {\ov \alpha} \int_s^t \io \big( \sigma\ee (\ln\sigma\ee - 1) 
   + \epsi (\sigma\ee - 1)^2 + 1 \big) \,\de x \, \de \tau
   \le {\ov \alpha} Z_\varepsilon, 
\end{equation}
for every $0 \le s < t \le T$. Of course, by Fatou's lemma, the above also implies
\begin{equation}\label{Zeta2}
  {\ov \alpha} \int_s^t \io \big( \sigma (\ln\sigma - 1) + 1 \big) \,\de x \, \de \tau
    \le {\ov \alpha} Z.
\end{equation}
Hence, collecting the above considerations and taking the $\liminf$ of \eqref{enre4},
we deduce
\begin{align}\nonumber
  & \calF(t) \le \lim_{\epsi\searrow 0} \calF\ee(s)
   - \int_s^t \calD(\tau) \, \de \tau  
   + \int_s^t \io \alpha(\fhi,\sigma) \big( \sigma \ln\sigma + \chi\sigma( 1 - \fhi) \big) \, \de x \, \de \tau\\
 \label{enre7}
  & \mbox{}~~~~~ + {\ov \alpha} \Big( Z - \int_s^t \io \big( \sigma (\ln\sigma - 1) + 1 \big) \,\de x \, \de \tau \Big),
\end{align}
where we point out that the last integral is nonnegative in view of \eqref{Zeta2}. 
 In general, we are not able to prove that, for every (or even for a.e.) $s\in (0,T)$,
$\lim_{\epsi\searrow 0} \calF\ee(s) = \calF(s)$. However, on account of 
\eqref{sigmaepsi0}-\eqref{sigmaepsi1}, using also the dominated convergence theorem,
one can easily prove that this holds at least for $s=0$. 
Hence, replacing $\calF$ with $\calE$ and rearranging terms, it is immediate
to get \eqref{enin}, which concludes the proof of the theorem.
\begin{remark}\label{rem:EvscalE}
  We point out that, if we denote as $E(t)$ the limit as $\epsi\searrow 0$ 
  of $\calF\ee(t)$, which exists for every $t\in [0,T]$ thanks to Helly's theorem,
  then the ``corrected'' energy functional $E(t)$ satisfies (by semicontinuity)
  $\calF(t) \le E(t)$ for every $t\in[0,T]$; moreover, $E$ satisfies the following
  alternative version of the energy inequality \eqref{enre7}:
 \begin{align}\nonumber
   & E(t) + \int_s^t \calD(\tau) \, \de \tau  
    \le E(s)
    + \int_s^t \io \alpha(\fhi,\sigma) \big( \sigma \ln\sigma + \chi\sigma( 1 - \fhi) \big) \, \de x \, \de \tau\\
  \label{enre7E}
   & \mbox{}~~~~~ + {\ov \alpha} \Big(  \int_s^t E (\tau) - \mathcal{F}(\tau) \, \de \tau \Big),
   \quad\text{for {\bf every} }\,0\le s < t \le T.
 \end{align}
\end{remark}


\subsection{Weak sequential stability of the solution set}
\label{subsec:stab}

As observed in Remark~\ref{rem:login}, we designed the solution concept 
(including, in particular, the entropy \eqref{weakentropy} and energy \eqref{enin} 
inequalities in the formulation), so to guarantee the weak sequential stability of 
the solution set. This is an important property, 
for instance for selection principles or further coupling of the model and, roughly speaking,
corresponds to the fact that any cluster point of any ``bounded''
family of weak solutions (in the sense of Def.~\ref{def:weak},
hence including the entropy and energy relations) is a weak solution again. 
This fact is made precise by the following result: 
\smallskip
\begin{proposition}\label{prop:stab}
  Let $\{(\varphi_n, \mu_n, \sigma_n)\}$ be a sequence of solutions 
  in the sense of Definition~\ref{def:weak}, all
  emanating from the same initial datum $ (\varphi_0,\sigma_0)$ under\/
  {\rm Hypothesis~\ref{hypo:data}}. Let $\{(\varphi_n, \mu_n, \sigma_n)\}$ satisfy
  the assumptions of Theorem~\ref{thm:global}, including 
  the entropy and energy inequalities\/ \eqref{weakentropy} and\/ \eqref{enin},
  with the (bounded) functions $\{\zeta_n\}$ and $\{Z_n\}$ therein being bounded by the 
  assigned problem parameters, hence independently of $n$. Moreover, 
  let us assume that, for some $c>0$ independent of $n$, there holds 
  the additional estimate
  \begin{equation}\label{add:ln}
     \big\| | \ln \sigma_n | \ln ( 1 + | \ln \sigma_n | ) \big\|_{L^\infty(0,T;L^1(\Omega))} \leq c.      
  \end{equation}
  Then, every cluster point of $\{(\varphi_n, \mu_n, \sigma_n)\}$ 
  is still a weak solution in the sense of Definition~\ref{def:weak}
  fulfilling in particular the regularity conditions
  \eqref{reg:mu}-\eqref{reg:nonl}.
\end{proposition}
\noindent%

\begin{remark}\label{rem:clust}
 It is worth clarifying a bit the meaning of the statement, which is somehow
 delicate as uniqueness is lacking and many weak solutions in the regularity
 class specified by Def.~\ref{def:weak} may emanate from a single initial
 datum. What we can show is that, if a family of such solutions also satisfies the entropy and energy 
 relations uniformly with respect to $n$ and with the functions $\zeta_n$ and $Z_n$
being bounded by the assigned system parameters,
  then such a family is weakly sequentially 
 stable, and in particular it also fulfills, uniformly in $n$, all the a-priori
 estimates corresponding to the regularity properties \eqref{reg:mu}-\eqref{reg:nonl}.
 This does not exclude that from the same initial datum there may also emanate
 other ``bad'' or ``unphysical'' weak solutions that do not satisfy the 
 entropy and energy inequalities in the quantitative way prescribed 
 by the statement.
 Moreover, the additional assumed estimate~\eqref{add:ln} is not required for 
 inferring the formulations of Definition~\ref{def:weak} but merely the
 first regularity statement of~\eqref{red:ln}, \textit{i.e.,} 
 $ \ln \sigma \in L^\infty(0,T;L^1(\Omega))$. 
\end{remark}
\begin{proof}
Let us assume that a triple $(\fhi,\mu,\sigma)$ arises as a cluster point, i.e.,
is a limit, in a suitable sense, of a (nonrelabelled) subsequence of 
$\{(\fhi_n,\mu_n,\sigma_n)\}$. We shall then prove that this ``cluster''
triplet still verifies all the conditions in Def.~\ref{def:weak}.
To this aim, we first observe that the mass balance properties at the 
level $n$ can be deduced by choosing $\psi=1$ in \eqref{weak:phi}
and, respectively, $\vartheta \equiv 1 $ in~\eqref{weak:sig}; it is 
also apparent that these conditions are maintained as we let 
(the subsequence of) $n$ go to $\infty$.

Next, as the entropy and energy inequalities \eqref{weakentropy} and \eqref{enin}
hold uniformly with respect to~$n$, we can recover basically
all the a priori estimates corresponding to the regularity properties 
\eqref{reg:mu}-\eqref{reg:nonl},  with the notable exception
of the last of \eqref{red:ln}, which will be discussed in a while.

Actually, one first uses the energy estimate \eqref{enin}, which, adapting
some regularity arguments from the previous section, provides most of the 
information. Then, choosing $\theta=1$ in \eqref{weakentropy} one deduces 
a priori estimates corresponding to the first two controls in \eqref{red:ln}.
To be precise, in order to get that 
\begin{equation}\label{ln:11}
  \| \ln \sigma_n \|_{L^\infty(0,T;L^1(\Omega))} \le c,  
\end{equation}
the entropy inequality should be stated on a generic subinterval $(0,t)$ 
(which is possible, as is observed in Remark~\ref{rem:weakstrong}).
Note, however, that \eqref{ln:11} alone does not exclude that 
the limit $\ln \sigma$ could exhibit concentration phenomena 
(i.e., be a non-absolutely continuous measure) with respect to space 
variables. 

As mentioned, obtaining the last estimate in \eqref{red:ln} is more
delicate as the BV-regularity of $\ln \sigma_n$ does not allow, at 
least directly, to integrate by parts back the corresponding terms in
\eqref{weakentropy} so to get an estimate of the time derivative.
To overcome this issue, we take advantage of the results proven in
Section~\ref{sec:pre}. In particular, using Lemma~\ref{lem:meas}, we infer
a measure-valued formulation which is equivalent to~\eqref{weakentropy}. Indeed, 
from the lemma, we deduce the existence of a family of  
measures~$\{\xi_n\}_{n\in \N}\subset \mathcal{M}^+(\ov\Omega \times [0,T])$ 
such that
 \begin{multline}\label{meastime}
    -\int_\Omega \ln \sigma_n \theta \, \de x \Big |_0^T 
    + \int_0^T \int_\Omega \theta | \nabla \ln \sigma_n|^2 
      - \nabla \ln \sigma_n \cdot \nabla \theta - \chi \theta \nabla\ln \sigma _n\cdot \nabla \varphi_n
        + \chi \nabla \varphi_n \cdot \nabla \theta \, \de x \, \de t\\
    + \int_0^T\int_{\ov\Omega} \theta \, \de \xi_n  (t,x)   
    =  \int_0^T\int_\Omega \alpha(\varphi_n,\sigma_n) \theta -\ln \sigma_n \t \theta \, \de x \, \de t 
\end{multline}
for all $ \theta \in \C^1( \ov\Omega \times [0,T];\R)$ with 
\[
  -\int_0^T \int_\Omega \theta \, \de \xi_n  (t,x)
    \leq \| \theta_-\|_{L^\infty(\Omega\times (0,T))}\left[  \zeta 
    - \int_0^T \int_\Omega| \nabla \ln \sigma_n |^2 \, \de x \, \de t  \right],
\]
where $\zeta$ is independent of $n$ as the initial data
are fixed. 
Now, a comparison in~\eqref{meastime} allows us to find an estimate for the 
time derivatives of~$\{ \ln \sigma_n\}$ via
\begin{align*}
   & - \int_0^T \langle \t \ln \sigma_n , \theta \rangle \, \de t 
    := -\langle \ln \sigma_n ,\theta \rangle  \Big |_0^T
      + \int_0^T\int_\Omega \ln \sigma_n \t \theta \, \de x \, \de t \\
   & \quad =  - \int_0^T \int_\Omega \theta | \nabla \ln \sigma_n|^2 
   - \nabla \ln \sigma_n \cdot \nabla \theta - \chi \theta \nabla\ln \sigma_n \cdot \nabla \varphi _n
   + \chi \nabla \varphi _n \cdot \nabla \theta \, \de x \, \de t \\
    &\qquad\quad- \int_0^T\int_{\ov\Omega} \theta \, \de \xi_n  (t,x)
     +  \int_0^T\int_\Omega \alpha(\varphi_n,\sigma_n) \theta \,  \de x \, \de t  \\
    & \quad \leq \left[\int_0^T \int_\Omega (1+\chi)|\nabla \ln \sigma_n|^2
     + \chi|\nabla \varphi_n|^2 \, \de x\, \de t+ \| \xi_n\|_{\mathcal{M}(\ov\Omega\times [0,T])}
      + \bar \alpha T|\Omega|\right]\| \theta \|_{L^\infty(\Omega\times (0,T))}  \\
    &\quad \quad \quad 
    +  \left[ \| \nabla \ln \sigma_n\|_{L^2(0,T;H)}
    + \chi T \| \nabla \varphi_n\|_{L^\infty(0,T;H)}\right]
    \| \nabla \theta\|_{L^2(0,T;H)}.
\end{align*}
The above relation implies that 
\begin{align*}
    \| \t \ln \sigma _n \|_{\mathcal{M}([0,T]; (W^{1,p}(\Omega))^*)} \leq C 
     \quad\text{for }\,p>d. 
\end{align*}
with $c>0$ independent of $n$. 
\smallskip

Having now all the a-priori estimates corresponding to \eqref{reg:mu}-\eqref{reg:nonl}
at our disposal, we now see that it is possible to pass to the limit 
as (a subsequence of) $n$ goes to $\infty$ and obtain 
that the ``limit'' triplet $(\fhi,\mu,\sigma)$ is still a weak solution 
in the sense of Def.~\ref{def:weak}.
Actually, most of the procedure closely follow the arguments
performed in the previous section for taking the limit $\epsi\searrow 0$
in the approximation scheme, and some parts of the argument 
are, in fact, even simpler as here we do not need to take care of 
the technical complications involving the nonlinear sensitivity. 

For this reason, we will just focus on the parts which involve some novelty.
First of all, we observe that the uniform integrability of (the negative 
part of) $\ln\sigma_n$, i.e., condition \eqref{add:ln},
is {\sl not}\/ a consequence of the energy and entropy relations.
Moreover, it cannot be reproduced (at least in a rigorous way) without referring
to a suitable regularization. Actually, the test function 
$\gamma(\ln \sigma_n)$ used before (where 
$\gamma$ is defined in \eqref{defi:gamma}), cannot be inserted in 
relation \eqref{weakentropy} due to its poor regularity properties. 
This is why \eqref{add:ln} is taken as an additional assumption
in the statement of Prop.~\ref{prop:stab}. Using \eqref{add:ln},
we can actually deduce the analogue of \eqref{co:14d},
which improves \eqref{ln:11} and excludes that the limit of 
$\ln\sigma_n$ could exhibit concentration phenomena.

With this information at disposal it is apparent that we can take 
the limit in equations~\eqref{weak:phi}-\eqref{weak:mu} and in the 
initial and boundary conditions \eqref{bou:fhi} as well. Hence, 
to conclude the proof,
it is sufficient to show that also the entropy and energy inequalities 
\eqref{weakentropy} and \eqref{enin} are stable with respect to $n\nearrow\infty$,
or, in other words, that they are still satisfied in the limit by the ``cluster''
solution $(\fhi,\mu,\sigma)$. 

Starting with the energy inequality \eqref{enin} (where we recall that the
real-valued numbers $Z_n$ are bounded independent of $n$ as the initial data are fixed), 
it is easy to see that the estimates obtained before and the very same semicontinuity
argument performed in Subsec.~\ref{subsec:enen} imply that it is possible to take 
$n\nearrow \infty$ therein and obtain an analogous inequality in the limit.
Finally, we need to show that the entropy inequality is also stable with
respect to $n$. Therefore, we again observe that the sequence $ \{ \zeta_n\} $
is a sequence of bounded numbers such that $ \zeta_n \to \zeta $ for a 
suitable subsequence. The limit passage in inequality~\eqref{weakentropy}
follows now in the same way as in Subsection~\ref{subsec:lim}. 
\end{proof}




\section{Existence of local-in-time strong solutions}
\label{sec:local}

In this part, we prove Theorem~\ref{thm:local}. Before that, however, we prefer to give a number of 
remarks aimed at clarifying the ``spirit'' of the local existence result.
\begin{remark}\label{rem:reg0}
 In the statement we implicitly assumed $d=3$, 
 and, for simplicity, also the proof will be carried out in 
 the three-dimensional setting. As already noticed, the 
 only notable difference in the case $d=2$ is represented
 by the (better) Sobolev exponents. 
\end{remark}
\begin{remark}\label{rem:reg1}
 It is worth observing that the regularity properties
 \eqref{strong:fhi}-\eqref{strong:mu},
 as well as the underlying assumption \eqref{init:reg1}, correspond, 
 in the Cahn-Hilliard terminology, to the outcome of the so-called ``second energy  estimate'',
 a procedure which is fully compatible with the  singular potential \eqref{Flog} and which, 
 as first observed in \cite{MiranvilleM2AN2004}, provides the maximal regularity available 
 so far for this type of system. Indeed, the main obstacle preventing the validity
 of a {\sl global}\/ regularity estimate comes from the supercritical 
 character of the quadratic cross diffusion term in 
 \eqref{eq:2} and not from the structure of system \eqref{eq:1}. 
\end{remark}
\begin{remark}\label{rem:reg2}
 As far as one looks for a {\sl local in time}\/ result, even in presence of 
 a singular potential of the form \eqref{Flog}, it is still possible to 
 further improve the regularity of $\fhi$, provided that the initial
 datum satisfies the following ``separation property":
 \begin{equation}\label{sep:prop}
   -1+\delta \le \fhi_0(x) \le 1-\delta
    \quad\text{for every }\, x\in \Omega\ \ \text{and some }\,
    \delta\in(0,1/2).
 \end{equation}  
 Actually, from the second of \eqref{strong:fhi}, one may deduce
 \begin{equation}\label{sep:prop2}
   \| \fhi(t) - \fhi_0 \|_V \le t^{1/2} \| \fhi_t \|_{L^2(0,T_0;V)}
    \le C_0 t^{1/2},
 \end{equation}  
 for every $t \in [0,T_0]$, where the constant $C_0$ only depends on the 
 $L^2(0,T_0;V)$-norm of $\fhi_t$. Combining the above with 
 the third of \eqref{strong:fhi} and using interpolation it
 is not difficult to infer
 \begin{equation}\label{sep:prop3}
   \| \fhi(t) - \fhi_0 \|_{L^\infty(\Omega)} \le C_1 t^{\eta},
 \end{equation}  
 where $\eta\in(0,1/2)$ is a computable exponent and $C_1$ 
 only depends on known quantities.  Recalling 
 \eqref{sep:prop}, we then readily deduce 
 \begin{equation}\label{sep:prop2b}
   -1+\delta/2 \le \fhi(x,t) \le 1-\delta/2
    \quad\text{for every }\, (x,t)\in \Omega \times [0,T_1],
 \end{equation}  
 where $T_1\in (0,T_0]$ is a computable time. Once the above 
 ``local separation property'' is achieved, $F'$ basically
 loses its singular character and can be treated as a ``smooth''
 function, which essentially implies that, over $(0,T_1)$, the regularity
 of the local solution might be improved at will, provided of course
 that the initial data are taken smooth enough.
\end{remark}
\begin{proof}
In the following part, we sketch the proof of Theorem~\ref{thm:local}. For the 
sake of brevity, we shall derive a number of additional a-priori estimates
by working directly on the equations of system \eqref{eq}
without referring to any specific regularization or approximation scheme. We actually 
believe that such a procedure, though formal, is more suitable for emphasizing
the real difficulties of the procedure without insisting on technicalities.
We also implicitly assume that the a-priori estimates already proved for 
weak solutions keep holding in the current setting, and we will use them
whenever needed. 

That said, we start with testing \eqref{eq:2} by $\sigma$ so to obtain
\begin{equation}\label{giu:b1}
  \frac12 \ddt \| \sigma \|^2 
   + \| \nabla \sigma \|^2 
   = \chi \io \sigma \nabla \sigma \cdot \nabla \fhi \, \de x
   + \io \alpha(\fhi,\sigma) \sigma^2 \, \de x.
\end{equation}
In order to control the first term on the right-hand side, we 
may observe that
\begin{align}\nonumber
  \chi \io \sigma \nabla \sigma \cdot \nabla \fhi \, \de x
   & \le c \| \sigma \|_{L^4(\Omega)} \| \nabla \sigma \| \| \nabla \fhi \|_{L^4(\Omega)}\\
 \nonumber
 & \le c \| \sigma \|_{L^6(\Omega)}^{3/4} \| \sigma \|^{1/4} \| \nabla \sigma \| 
  \| \fhi \|_{L^\infty(\Omega)}^{1/2} \| \fhi \|_{H^2(\Omega)}^{1/2}\\
 \label{giu:b2}
 & \le c \| \sigma \|_{V}^{7/4} \| \sigma \|^{1/4} 
    \| \fhi \|_{H^2(\Omega)}^{1/2},
\end{align}
where we have used elementary interpolation of $L^p$-spaces and a well-known
Gagliardo-Nirenberg inequality. In order to control the last term, the simplest 
way consists in testing the second of \eqref{eq:1} by $-\Delta\fhi$ so to
obtain
\begin{equation}\label{giu:b3}
  \| \Delta \fhi \|^2 
   \le \lambda \| \nabla \fhi \|^2
   - \chi \io \sigma \Delta\fhi \, \de x
   + \io \nabla \mu \cdot \nabla\fhi \, \de x,
\end{equation}
where we have integrated by parts the semilinear term and used the $\lambda$-monotonicity
of $F'$ (we have already observed that this argument works in spite of the ``nonsmoothness''
of $F'$). Hence, using \eqref{reg:phi} and well-known elliptic regularity results,
it is not difficult to deduce
\begin{equation}\label{giu:b4}
  \| \fhi \|^2_{H^2(\Omega)}
   \le c \big( 1 + \| \Delta \fhi \|^2 \big)
   \le c \big( 1 + \| \sigma \|^2 \big)
   + c \| \nabla \mu \|.
\end{equation}
Replacing the above into \eqref{giu:b2}, we obtain
\begin{equation}\label{giu:b5}
  \chi \io \sigma \nabla \sigma \cdot \nabla \fhi \, \de x 
   \le c \| \sigma \|_{V}^{7/4} \| \sigma \|^{1/4} 
       \big( 1 + \| \sigma \|^{1/2} + \| \nabla \mu \|^{1/4} \big).
\end{equation}
Plugging this relation into \eqref{giu:b1}, employing the boundedness of $\alpha$
(cf.\ Hypothesis~\ref{hypo:weak}), and adding $\| \sigma \|^2$ to both-hand sides
so to recover the full $V$-norm on the left-hand side, we deduce
\begin{equation}\label{giu:b6}
 \frac12 \ddt \| \sigma \|^2 
   + \| \sigma \|^2_V
   \le c \| \sigma \|_{V}^{7/4} \| \sigma \|^{1/4} 
       \big( 1 + \| \sigma \|^{1/2} + \| \nabla \mu \|^{1/4} \big)
  + c \| \sigma \|^2.
\end{equation}
Let us now set $2 Y(t):=\| \sigma(t) \|^2$,
$m(t):=\| \nabla \mu (t) \|^2$ and notice that, rearranging
terms, \eqref{giu:b6} implies the following differential inequality:
\begin{align}\nonumber
  & Y'(t) + \| \sigma(t) \|^2_V \\
 \label{giu:b7}
  & \mbox{}~~~~~ \le c \big( \| \sigma(t) \|_{V}^{7/4} Y(t)^{1/8}
   + \| \sigma(t) \|_{V}^{7/4} Y(t)^{3/8}
   + \| \sigma(t) \|_{V}^{7/4} Y(t)^{1/8} m(t)^{1/8}
   + Y(t) \big).
\end{align}
Then, applying Young's inequality, we infer
\begin{equation}\label{giu:b7b}
 Y'(t) + \frac12 \| \sigma(t) \|^2_V 
   \le c \big( 1  + Y^3(t) \big) + \tilde{c} m(t) Y(t),
\end{equation}
where $\tilde{c}>0$ behaves like the generic $c$. Let us now set 
\begin{equation}\label{giu:b8}
   M(t):= \tilde{c} \int_0^t m(s) \, \de s
\end{equation}
and observe that, by \eqref{reg:mu},
\begin{equation}\label{giu:b9}
  \kappa \le e^{-M(t)} \le 1 \quad\text{for all }\,t\in[0,T]
\end{equation}
and for some $\kappa\in(0,1]$ depending only on the problem data.
Then, setting $Z(t):=e^{-M(t)}Y(t)$, it is not difficult to deduce 
from \eqref{giu:b7b} that
\begin{equation}\label{giu:b10}
 Z'(t)
   \le c ( 1  + Z^3(t) ).
\end{equation}
Applying the comparison principle for ODE's, we then deduce
that for a suitable, and computable, $T_0\in(0,T]$ depending on the problem data,
there holds
\begin{equation}\label{loc:11}
  \| \sigma \|_{L^\infty(0,T_0;H)} 
    + \| \sigma \|_{L^2(0,T_0;V)} \le c,
\end{equation}
where the second estimate also considers the dissipative contribution
on the left-hand side of \eqref{giu:b7b}. The above information 
also has some effect on the regularity of $\fhi$. Indeed, going
back to the second \eqref{eq:1}, we may interpret it as a time-dependent
family of elliptic problems; namely, we have 
\begin{equation}\label{giu:b12}
  - \Delta \fhi(t) + \beta(\fhi(t)) =
   h(t) := \lambda\fhi(t) + \mu(t) + \chi \sigma(t),
\end{equation}
where, by Sobolev's embeddings, it is a standard matter 
to verify that $h\in L^2(0,T_0;L^6(\Omega))$ (with the exponent
$6$ replaced by any $P\in[1,\infty)$ in the two-dimensional case).
Then, a well-known elliptic regularity argument implies
\begin{equation}\label{loc:12}
  \| \fhi \|_{L^2(0,T_0;W^{2,6}(\Omega))} \le c.
\end{equation}
To be more precise, the above estimate is obtained by (formally)
testing the second \eqref{eq:1} by $\beta(\fhi)^5$, 
which is a monotone function of its argument as a direct check shows,
so that we can use in particular the analogue of \eqref{giu:a1}.
Then, the procedure implies
that 
\begin{equation}\label{loc:12a}
  \| \beta(\fhi) \|_{L^2(0,T_0;L^6(\Omega))} \le c.
\end{equation}
Then, with the above information at hand, regularity results of Agmon-Douglis-Nirenberg 
type readily entail \eqref{loc:12}. This relation, in turn, can help improve the
regularity of $\sigma$. Indeed, we may rewrite \eqref{eq:2} in the form
\begin{equation}\label{eq:2:rev}
  \sigma_t - \Delta \sigma = \alpha(\fhi,\sigma) \sigma 
   - \chi \sigma \Delta\fhi - \chi \nabla \sigma \cdot \nabla \fhi,
\end{equation}
and we may apply suitable regularity results to the above equation.
In particular, we perform an additional ``parabolic'' estimate
by testing \eqref{eq:2:rev} by $\sigma_t$. Simple calculations then yield
\begin{align}\nonumber
  \| \sigma_t \|^2
   + \frac12 \ddt \| \nabla\sigma \|^2 
   & = \io \alpha(\fhi,\sigma) \sigma \sigma_t \, \de x
    - \chi \io \sigma_t \sigma \Delta\fhi \, \de x 
    - \chi \io \sigma_t \nabla\sigma\cdot \nabla \fhi \, \de x\\
 \label{giu:c1}  
  & =: I_1 + I_2 + I_3.
\end{align}
Thanks to Sobolev's embeddings, the
terms on the right-hand side can be controlled as follows:
\begin{align}\label{giu:c2}
  | I_1 | & \le c \| \sigma \| \| \sigma_t \|
   \le \frac14 \| \sigma_t \|^2 + c \| \sigma \|^2,\\
 \label{giu:c3}
  | I_2 | & \le c \| \sigma \|_{L^4(\Omega)} \| \sigma_t \| \| \Delta\fhi \|_{L^4(\Omega)}
   \le \frac14 \| \sigma_t \|^2 + c \| \sigma \|_{V}^2 \| \Delta\fhi \|_{L^4(\Omega)}^2,\\
 \label{giu:c4}
  | I_3 | & \le c \| \sigma_t \| \| \nabla\sigma \| \| \nabla \fhi \|_{L^\infty(\Omega)}
   \le \frac14 \| \sigma_t \|^2 + c \| \nabla\sigma \|^2 \| \fhi \|_{W^{2,6}(\Omega)}^2.
\end{align}
Collecting the above computations, \eqref{giu:c1} gives
\begin{equation}\label{giu:c5}
  \frac14 \| \sigma_t \|^2
   + \frac12 \ddt \| \nabla\sigma \|^2 
   \le c \| \sigma \|^2 + c \| \fhi \|_{W^{2,6}(\Omega)}^2 \| \sigma \|_V^2.
\end{equation}
Then, adding to both sides of the above relation the 
inequality 
\begin{equation}\label{giu:c6}
  \frac12 \ddt \| \sigma \|^2
    \le \frac18 \| \sigma_t \|^2
    + c \| \sigma \|^2,
\end{equation}
so to recover the full $V$-norm on the left-hand side (under the time
derivative), and applying
Gr\"onwall's lemma to the resulting relation, it is a standard matter to 
deduce the additional bounds
\begin{equation}\label{giu:c7}
  \| \sigma \|_{H^1(0,T_0;H)}
   + \| \sigma \|_{L^\infty(0,T_0;V)}
   \le c,
\end{equation}
where we have used \eqref{loc:12} and the fact 
$\sigma_0\in V$, which is a consequence of \eqref{init:reg2}.

Thanks to the above estimates, going back to \eqref{eq:2:rev},
we may notice that the last two terms on the right-hand side 
can now be controlled as follows:
\begin{align}\label{giu:c8}
    \| \sigma \Delta \fhi \|_{L^2(0,T_0;H)}
   \le c \| \sigma \|_{L^\infty(0,T_0;L^3(\Omega))}
     \| \Delta\fhi \|_{L^2(0,T_0;L^6(\Omega))}
      \le c,\\
 \label{giu:c9}
  \| \nabla \sigma \cdot \nabla \fhi \|_{L^{2}(0,T_0;H)}
   \le c \| \nabla\sigma \|_{L^\infty(0,T_0;H)}
     \| \nabla\fhi \|_{L^2(0,T_0;L^{\infty}(\Omega))}
      \le c,
\end{align}
where we have also used the continuous embedding 
$W^{2,6}(\Omega)\subset W^{1,\infty}(\Omega)$.
Taking \eqref{giu:c8}-\eqref{giu:c9} into account, comparing
terms and applying elliptic regularity results in \eqref{eq:2:rev},
we deduce
\begin{equation}\label{giu:c7b}
  \| \sigma \|_{L^2(0,T_0;H^2(\Omega))}
   \le c.
\end{equation}
To proceed, we need to observe that properties \eqref{giu:c7}-\eqref{giu:c7b} 
seem not yet sufficient to perform the weak-strong uniqueness argument.
Indeed, the above conditions, at least for $d=3$, do not imply the uniform boundedness 
of $\sigma$, which is used in an essential way in the proof.

In order to further improve the smoothness of solutions, we then need to go back to
the Cahn-Hilliard system and perform the so-called {\sl second energy estimate}.
Namely, we test the first \eqref{eq:1} by $\mu_t$ and sum the result
to the time derivative of the second \eqref{eq:1} tested by $\fhi_t$. 
We omit giving the details since the argument is rather standard
and one can refer to several papers (see, e.g., \cite[Sec.~3]{MiranvilleM2AN2004})
where it is carried out extensively. We just point out that,
in the current regularity setting, the argument may have a formal character,
but, on the other hand, it could be easily justified by approximation. 
In particular, the procedure is fully compatible with the occurrence of the singular 
potential as, in fact, it only relies on the $\lambda$-convexity of $F$. 
Moreover, we observe that the first of \eqref{giu:c7} is sufficient
to manage the coupling term depending on $\sigma$ (indeed, for that aim
what is needed is the $L^2$-regularity of $\sigma_t$ which is part
of \eqref{giu:c7}). That said, using also the condition \eqref{init:reg1}
on the initial data, we deduce
\begin{equation}\label{giu:c10}
  \| \fhi \|_{W^{1,\infty}(0,T_0;V')}
  + \| \fhi \|_{H^1(0,T_0;V)}
   + \| \mu \|_{L^\infty(0,T_0;V)}
   \le c.
\end{equation}
Next, again owing to Sobolev's embeddings, applying elliptic regularity results
to the second \eqref{eq:1}, and using the second of \eqref{giu:c7}, 
we infer
\begin{equation}\label{giu:c11}
  \| \fhi \|_{L^\infty(0,T_0;W^{2,6}(\Omega))}
  + \| \beta(\fhi) \|_{L^\infty(0,T_0;L^6(\Omega))}
   \le c,
\end{equation}
where, as before, the exponents $6$ may be replaced by any $P\in[1,\infty)$ 
in the two-dimensional case. Next, as a consequence of 
\eqref{giu:c10}-\eqref{giu:c11}, we also have
\begin{align}\label{giu:c8b}
  \| \nabla \sigma \cdot \nabla \fhi \|_{L^{4}(0,T_0;L^3(\Omega))}
   \le c \| \nabla\sigma \|_{L^4(0,T_0;L^3(\Omega))}
     \| \nabla\fhi \|_{L^\infty(0,T_0;L^{\infty}(\Omega))}
      \le c,\\
 \label{giu:c9b}
    \| \sigma \Delta \fhi \|_{L^\infty(0,T_0;L^3(\Omega))}
   \le c \| \sigma \|_{L^\infty(0,T_0;L^6(\Omega))}
     \| \Delta\fhi \|_{L^\infty(0,T_0;L^6(\Omega))}
      \le c.
\end{align}
In particular, it may be worth remarking that, when deducing relation \eqref{giu:c8b}, 
we have applied to $\nabla\sigma$ the continuous embedding
\begin{equation}\label{giu:b15}
   L^\infty(0,T_0;H) \cap L^2(0,T_0;V) 
    \subset L^4(0,T_0;L^3(\Omega)).
\end{equation}
Then, by Sobolev's embeddings,
\eqref{init:reg2} implies in particular
\begin{equation}\label{besov:11}
   \sigma_0\in W^{3/2,3}(\Omega)
   \subset B^{3/2}_{3,3}(\Omega)
   \subset B^{3/2}_{3,4}(\Omega)
\end{equation}
(see, e.g., \cite[Thm.~6.2.4 and Thm.~6.4.4]{Bergh}).
The above condition allows us to apply to equation \eqref{eq:2:rev}
inhomogeneous (i.e., of the so-called $L^p$-$L^q$ type, 
here with $p=4$, $q=3$) parabolic regularity results of
Agmon-Douglis-Nirenberg type (see, 
\textit{e.g.,}~\cite[Thm.~2.1]{Denk2} or~\cite{Denk}), 
implying
\begin{equation}\label{giu:c12}
  \| \sigma_t \|_{L^4(0,T_0;L^3(\Omega))}
   + \| \sigma \|_{L^4(0,T_0;W^{2,3}(\Omega))}
   \le c.
\end{equation}
Hence, noting the interpolation property (see, e.g., \cite[Par.~6.4]{Bergh},
for the properties of the real interpolation operator
$(\cdot,\cdot)_{\theta,q}$ in the framework of Sobolev and Besov 
spaces)
\begin{equation}\label{giu:c12x}
  \big( W^{1,4}(0,T_0;L^3(\Omega)), L^4(0,T_0;W^{2,3}(\Omega))\big)_{1/3,q}
    \subset L^\infty(0,T_0;L^\infty(\Omega)),
\end{equation}
which holds for every $q\in(1,\infty)$ thanks to embedding properties of Sobolev
and Besov spaces, from \eqref{giu:c12}

we eventually deduce
\begin{equation}\label{giu:c13}
  \| \sigma \|_{L^\infty(0,T_0;L^\infty(\Omega))}
     \le c.
\end{equation}
This concludes the proof of the theorem as the regularity properties
\eqref{strong:fhi}-\eqref{strong:sigma2} are direct consequences of 
\eqref{giu:c10}, \eqref{giu:c11}, \eqref{giu:c12} and \eqref{giu:c13}.
In particular, it is immediate to check that the summability of the second
space derivatives both of $\fhi$ and of $\sigma$, thanks to suitable 
trace theorems, allows us to interpret both the system equations 
and the boundary conditions in the pointwise sense.

\end{proof}


\section{Relative energy inequality and weak-strong uniqueness}
\label{sec:weakstrong}

In the following argument, using a notation which is rather customary in weak-strong
uniqueness proofs, we shall rename as $(\tp,\tmu,\ts)$ the local strong 
solution provided by Theorem~\ref{thm:local}. Moreover,
we will note as $V_0$ and $V_0'$ the closed subspaces of $V$ and, 
respectively, $V'$, consisting of the function(al)s having zero spatial mean. 
Then, $-\Delta$ (i.e., minus the Neumann Laplacian) 
turns out to be invertible as an operator from $V_0$ to $V_0'$
and we will denote as $(-\Delta)^{-1} : V_0' \to V_0$ its inverse mapping. 
We also point out that, respectively on $V_0$ and $V_0'$, the norms 
\begin{equation}\label{norms0}
  \| v \|_{V_0}^2 := \| \nabla v \|^2, \qquad
  \| \xi \|_{V_0'}^2 := \langle \xi, ( - \Delta)^{-1} \xi \rangle
\end{equation}
are equivalent to the standard norms. We will use the above norms whenever
necessary. 

That said, the {\sl relative energy} $\mathcal R$ is defined as 
\begin{align*}
   \mathcal R( \varphi , \sigma| \ts,\tp ) 
    := \int_\Omega \sigma - \ts-\ts\ln \left ( \frac{ \sigma}{\ts}\right) \de x 
    + \frac{M}{2}\| \varphi - \tp \|_{V'_0}^2,
\end{align*}
and the {\sl relative dissipation}\/ by 
\begin{align*}
   \mathcal{W}( \varphi , \sigma| \ts,\tp ) 
     & := \int_\Omega\ts | \nabla \ln \sigma - \nabla \ln \ts|^2 
       - \chi \ts (\nabla \ln \sigma - \nabla \ln \ts) \cdot (\nabla \varphi - \nabla \tp) \, \de x \\
    & \quad + M\int_\Omega | \nabla \varphi -\nabla \tp |^2 
      + (F'(\varphi)-F'(\tp))(\varphi-\tp) + \lambda | \varphi - \tp|^2 \, \de x,
\end{align*}
where $M>0$ will be chosen later on. 
Then, recalling that $\tmu $ is 
given by $ \tmu := -\Delta \tp + F'(\tp) - \chi \ts$,
we (formally) introduce the operator
\begin{align*}
 \mathcal{A}( \tp , \ts) = \begin{pmatrix}
  \t \tp - \Delta \tmu \\ 
  \t \ts - \Delta \ts - \chi \dive ( \ts \nabla ( 1-\tp)) 
   -\alpha(\tp,\ts) \ts 
\end{pmatrix},
\end{align*}
so that the system can be abstractly reformulated as $\mathcal{A}( \tp , \ts)  = 0$.

With these definitions at hand, we are able to state the relative
energy inequality.
Actually, as $ \varphi(0)= \tp(0)$, the mass conservation
property gives
\begin{equation}\label{stesmed}
  \int_\Omega \varphi (t)\, \de x = \int_\Omega \tp(t)\, \de x
  \quad \text{for all }\,t>0;
\end{equation}
  namely, $ \varphi (t) -\tp(t)\in V_0$ for every $t\in [0,T]$.  
Then, using Ehrling's lemma, we also observe, for later convenience, that
\begin{align}\label{interpolation}
  \| \varphi - \tp\|^2 \leq \eta \| \nabla \varphi - \nabla \tp\|^2 
   + C_{\eta} \| \varphi-\tp\|_{V'_0}^2
\end{align}
for a possibly small constant $\eta>0$ and a correspondingly large constant $C_\eta$. 
\begin{proposition}\label{prop:relen}
 Let $(\varphi, \sigma)$ be a weak solution in the sense of 
 Definition~\ref{def:weak} and $(\tp,\ts)$ a pair of functions 
   satisfying, over the interval $(0,T)$, regularity 
 conditions analogous to~\eqref{weak:reg}, \eqref{strong:fhi}, 
 and \eqref{strong:sigma}-\eqref{strong:sigma2}. 
 %
%
 Then, the {\sl relative energy inequality}
 \begin{multline}
   \frac{\de}{\de t}  \mathcal{R}( \varphi,\sigma | \tp , \ts)
      + \mathcal{W}( \varphi , \sigma  | \tp , \ts  )
  %
   \le  \int_\Omega \left[\alpha(\varphi,\sigma)- \alpha(\tp,\ts)\right] ( \sigma-\ts)
      + \alpha(\tp,\ts) \left[ \sigma-\ts- \ts\ln 
    \Big(\frac{\sigma}{\ts}\Big) \right] \, \de x \\
   + \int_\Omega M \chi (\sigma - \ts) ( \varphi - \tp ) 
    + \lambda M | \varphi - \tp|^2 \, \de x \, 
  \label{relenin}
 \end{multline}
 is fulfilled for a.e. $0\leq  t \leq T$.
 %
\end{proposition}
\begin{proof}
First of all, we observe that
%
\begin{align}\label{relen}
    \frac{\de}{\de t} \int_\Omega \sigma - \ts-\ts\ln \left ( \frac{ \sigma}{\ts}\right) \, \de x 
    =  \frac{\de}{\de t}\langle  \sigma, 1 \rangle  
        - \langle \t \ln \sigma , \ts\rangle - \int_\Omega \t \ts ( \ln \sigma - \ln \ts ) \, \de x . 
\end{align}
Next, integrating equation~\eqref{weak:sig} over $\Omega $ 
and choosing $\vartheta \equiv 1$, we infer 
\begin{equation}\label{eq:volume}
    \frac{\de}{\de t}  \int_\Omega \sigma \, \de x 
   = \int_\Omega \alpha(\varphi, \sigma) \sigma \, \de x. 
\end{equation}
Now, let us insert~$\theta=\ts \phi$ in~\eqref{weakentropy} 
with $\phi \in \C^1([0,T))$ and $\phi \geq 0$.   To be precise,
it may be worth noting that, under the conditions on $\ts$
\eqref{strong:sigma}-\eqref{strong:sigma}
assumed here, such a function $\theta$ does not properly belong 
to $C^1([0,T];W^{1,p}(\Omega))$, $p>d$, as would required in 
Def.~\ref{def:weak}; on the other hand, one can easily 
prove that the validity of relation \eqref{weakentropy} can be 
extended to test functions $\theta$ with the (weaker)
regularity corresponding to \eqref{strong:sigma}-\eqref{strong:sigma2}.
In addition to that, as noted in Remark~\ref{rem:reg2}, the 
regularity of $\ts$ (hence that of~$\theta=\ts \phi$) could be
further improved at the price of technical complications. 
That said,  we infer via Lemma~\ref{lem:invar} that
\begin{align}\nonumber
   - 
   \langle\t \ln \sigma , \ts\rangle_{W^{1,p}(\Omega)} \, 
   + 
   \int_\Omega \ts | \nabla \ln \sigma|^2 
   - \nabla \ln \sigma \cdot \nabla \ts - \chi \ts \nabla\ln \sigma \cdot \nabla \varphi 
   + \chi \nabla \varphi \cdot \nabla \ts \, \de x &\,
   \\
  \label{weakentropy2}
   \qquad\quad
  + 
  \int_\Omega \alpha(\varphi , \sigma)  \ts 
     \, 
     \de x \, 
    & \leq 0.
\end{align}
Next, we test \eqref{eq:2} (written for $(\ts,\tp)$)  by
$\ln \sigma - \ln \ts$. We note that this test is indeed justified due to 
the additional regularity assumed for $\ts$. Straightforward computations lead to 
\begin{align}\nonumber
   \int_\Omega \t \ts ( \ln \sigma - \ln \ts ) \, \de x 
    + \int_\Omega (\nabla \ts - \chi \ts \nabla \tp)\cdot ( \nabla \ln \sigma - \nabla \ln \ts ) \, \de x &\\
 \label{ws:10}   
   \qquad = \int_\Omega \alpha(\tp,\ts) \ts ( \ln \sigma-\ln \ts) \, \de x.&
\end{align}
Summing~\eqref{eq:volume}  with \eqref{weakentropy2},
subtracting \eqref{ws:10}, and using~\eqref{relen},
we then deduce
\begin{multline}\label{eq:ineqweak1}
  \frac{\de}{\de t}   \int_{\Omega} \sigma -\ts -\ts\ln\frac{\sigma}{\ts}\, \de x 
     + \int_\Omega \ts| \nabla \ln\sigma-\nabla \ln \ts|^2 
         - \chi \ts (\nabla \ln \sigma - \nabla \ln \ts) \cdot (\nabla \varphi -\nabla \tp) \, \de x \\
    \leq \int_\Omega \alpha(\varphi,\sigma) ( \sigma-\ts)- \alpha(\tp,\ts) \ts
      \ln \frac{\sigma}{\ts} \, \de x \\
    = \int_\Omega \left[\alpha(\varphi,\sigma)- \alpha(\tp,\ts)\right] ( \sigma-\ts)
         + \alpha(\tp,\ts) \left[ \sigma-\ts- \ts\ln \frac{\sigma}{\ts} \right] \, \de x.
\end{multline}
Above, we also used the fact that $ \t \ts \ln \ts = \t ( \ts (\ln \ts - 1))$. 
Next, we write \eqref{weak:phi} for $\fhi$ and for $\tp$, take the difference, and choose
$\psi=(-\Delta)^{-1}(\varphi-\tp)$ therein, which is allowed thanks to \eqref{stesmed}.
Then, some manipulations lead to
\begin{equation}\label{ws:11}
  \frac12 \ddt \| \fhi - \tp \|_{V_0'}^2
   + \io (\mu - \tmu) (\varphi-\tp) \, \de x
   = 0.
\end{equation}
Correspondingly, we compute the difference of \eqref{weak:mu} for the weak and the strong 
solution, test it by $\varphi-\tp$, and integrate over $\Omega$ so to obtain
\begin{equation}\label{ws:12}
  \io (\mu - \tmu) (\varphi-\tp) \, \de x 
   = \| \nabla \varphi - \nabla \tp \|^2
   + \io \big( ( F'(\fhi) - F'(\tp) - \chi (\sigma - \ts) \big) (\varphi-\tp) \, \de x .
\end{equation}
Then, replacing \eqref{ws:12} into \eqref{ws:11} we are led to 
\begin{align}\nonumber
   & \frac{1}{2}\frac{\de}{\de t }\| \varphi - \tp\|_{V_0'}^2 
   +
   \int_\Omega | \nabla \varphi - \nabla \tp|^2 
     + ( F'(\varphi)-F'(\tp))(\varphi-\tp) \, \de x \, 
     \\
 \label{ineq:ch}
   & \quad  = \chi 
   \int_\Omega ( \sigma-\ts)(\varphi-\tp) \, \de x \, 
   .
\end{align}
Adding~\eqref{ineq:ch} multiplied by $M$ to~\eqref{eq:ineqweak1} 
and adding   $ M \lambda \| \varphi - \tp\|^2$ 
to both sides, we infer the asserted inequality~\eqref{relenin}. 
\end{proof}
\begin{corollary}\label{cor:weakstrong}
 Let $(\varphi, \sigma)$ be a weak solution in the sense of Definition~\ref{def:weak} 
 and $(\tp,\ts)$ be a strong solution in the sense of Theorem~\ref{thm:local},
 such that $ (\varphi(s),\sigma(s)) = ( \tp(s), \ts(s))$ for some 
 $s \in [0,T_0)$. Then, it holds that 
 $(\varphi(t),\sigma(t)) = ( \tp(t), \ts(t))$ for all $t\in [s,T_0)$. 
\end{corollary}
\begin{proof}
First of all, let us choose 
$M := \max\{ 1, \chi^2 \| {\ts}\|_{L^\infty(\Omega\times (0,T))} \}$. 
Then, it is a standard matter to check that
$ \mathcal{W}$ is nonnegative. More precisely, by Young's inequality, we have 
\begin{align}
  & \int_\Omega\ts | \nabla \ln \sigma - \nabla \ln \ts |^2 
     + \chi \ts (\nabla \ln \sigma - \nabla \ln \ts) \cdot (\nabla \varphi - \nabla \tp) 
       + M | \nabla \varphi -\nabla \tp|^2 \, \de x\notag  \\
 \label{nonnegative} 
  & \qquad\quad  \geq \frac 1 2 \int_\Omega\ts | \nabla \ln \sigma - \nabla \ln \ts|^2 
  + M | \nabla \varphi -\nabla \tp|^2 \, \de x.
\end{align}
Then, applying Lemma~\ref{lem:fenchel_young}, we get that 
\begin{equation}\label{wsu:11}
        (\sigma-\ts)(\varphi-\tp)\leq 4 \left( \ts | \varphi - \tp| ^2 
          + \sigma - \ts - \ts\ln \left(\frac{\sigma}{\ts}\right)\right),
\end{equation}
as well as 
\begin{align}\nonumber
  & \left[\alpha(\varphi,\sigma)- \alpha(\tp,\ts)\right] ( \sigma-\ts) \\
 \nonumber
  & \qquad \leq 4\max \left\{ 1 , | \ov\alpha |, | \underline\alpha | \right\} 
         \left( \ts\left[\alpha(\varphi,\sigma)- \alpha(\tp,\ts)\right]^2 
         + \sigma - \ts- \ts \ln \frac{\sigma}{\ts}  \right) \\
 \label{wsu:12}
  & \qquad\quad \leq  \CA \left( | \varphi -\tp|^2 
      + | \sqrt{\sigma }- \sqrt{\ts}|^2
    + \sigma - \ts- \ts \ln \frac{\sigma}{\ts}  \right),
\end{align}
where \eqref{ass:alpha} has been used, and the constant(s) $\CA$ 
(whose value is allow to vary on occurrence) may also depend on 
the $L^\infty$-norm of $\ts$. 

  Now, noting that the estimate of the right-hand side of \eqref{wsu:12}
is straighforward,  we may insert \eqref{wsu:11} and \eqref{wsu:12}
into \eqref{relenin}. Hence, in view of estimate~\eqref{interpolation}, 
it only remains to estimate the difference of the square-roots 
of $\sigma$ and $\ts$. To this end, we first observe that the term 
$ \Lambda (u | \tilde u ) :=  e ^ {\frac{1}{2}u} - e^{\frac{1}{2}\tilde u }
   - \frac{1}{2}e^{\frac{1}{2}\tilde{u}} (u - \tilde u )$ 
is nonegative for all $u,\tilde{u} \in \R$ since $x\to e^{\frac{1}{2}x}$ is a convex 
function and secondly that 
$$
  0\leq  \Lambda(\ln \sigma | \ln \ts) 
    = \sqrt{\sigma} - \sqrt{\ts} - \frac{1}{2} \sqrt{\ts}(\ln \sigma- \ln \ts).
$$
By elementary considerations, we then obtain
\begin{align}\label{wsu:13}
  (\sqrt{\sigma}-\sqrt{\ts})^2 \leq  \sigma - 2 \sqrt{\sigma}\sqrt{\ts} + \ts 
    + 2\sqrt{\ts} \Lambda(\ln \sigma| \ln \ts) = \sigma - \ts - \ts (\ln \sigma - \ln \ts) ,
\end{align}
which can now be used to  estimate the difference of 
the square roots of $\sigma$ and $\ts$ appropriately. 
 
Indeed, under the assumptions of this corollary,
we find by integrating~\eqref{relenin} over $(s,t)$ 
and using~\eqref{wsu:13} that
\begin{multline}\label{releninest2}
  \mathcal{R}( \varphi,\sigma | \tp , \ts) \Big |_s^t 
    + \int_s^t \mathcal{W}( \varphi , \sigma  | \tp , \ts  ) \, \de \tau \\
    \leq \int_s^t \int_\Omega C_{\alpha,M,1}
    \Big( \sigma - \ts - \ts \ln \Big(\frac{\sigma}{\ts} \Big) \Big)
     + C_{\alpha,M,2} | \varphi - \tp|^2 \, \de x \, \de \tau, 
\end{multline}
  for suitable constants $C_{\alpha,M,1},C_{\alpha,M,2}>0$.  
Then, applying once more Ehrling's lemma, we deduce
$$
   C_{\alpha,M,2} \| \fhi - \tp \| ^2
     \leq \frac{M}{4} \int_{\Omega} | \nabla \varphi -\nabla \tp|^2 \, \de x
      + C_{\alpha,M,3} \| \varphi - \tp \|_{V'_0}^2 .
$$
Inserting this and~\eqref{nonnegative} into~\eqref{releninest2} implies 
\begin{equation*}
   \mathcal{R}( \varphi,\sigma | \tp , \ts) \Big |_s^t 
    + \frac{1}{4} \int_s^t \int_\Omega\ts | \nabla \ln \sigma - \nabla \ln \ts|^2  
        + M | \nabla \varphi -\nabla \tp|^2 \, \de x \, \de \tau 
     \leq C \int_s^t \mathcal{R}( \varphi,\sigma | \tp , \ts) \, \de \tau,
\end{equation*}
so that the assert follows directly from Gr\"onwall's lemma. 
\end{proof}



\section*{Acknowledgment}

ER and GS gratefully acknowledge  the support of the GNAMPA (Gruppo Nazionale per l’Analisi Matematica, la Probabilit\`a
e le loro Applicazioni) of INdAM (Istituto Nazionale di Alta Ma\-te\-ma\-ti\-ca).
ER also acknowledges the support of Next Generation EU Project number~P2022Z7ZAJ 
(“A unitary mathematical framework for modelling muscular dystrophies'').
RL acknowledges funding by the Deutsche Forschungsgemeinschaft (DFG, German Research Foundation) 
within SPP 2410 “Hyperbolic Balance Laws in Fluid Mechanics: Complexity, Scales, Randomness (CoScaRa)'',
project number 526018747 and the hospitality of the Erwin--Schrödinger institute in Vienna during the thematic semester on free boundary problems.






\bibliographystyle{abbrv}

\begin{thebibliography}{10}

\bibitem{agg}
H.~Abels, H.~Garcke, and G.~Gr\"un.
\newblock Thermodynamically consistent, frame indifferent diffuse interface models for incompressible two-phase flows with different densities.
\newblock \textit{ Math. Models Methods Appl. Sci.}, 22(3):1150013, 40, 2012.

\bibitem{AS}
A.~Agosti and A.~Signori.
\newblock Analysis of a multi-species {Cahn}-{Hilliard}-{Keller}-{Segel} tumor growth model with chemotaxis and angiogenesis.
\newblock \textit{ J. Differ. Equations}, 403:308--367, 2024.

\bibitem{barbu}
V.~Barbu.
\newblock \textit{ Nonlinear semigroups and differential equations in {B}anach spaces}.
\newblock Editura Academiei Republicii Socialiste Rom\^ania, Bucharest; Noordhoff International Publishing, Leiden, 1976.
\newblock Translated from the Romanian.

\bibitem{Bergh}
J.~Bergh and J.~L{\"o}fstr{\"o}m.
\newblock \textit{ Interpolation spaces. {An} introduction}, volume 223 of \textit{ Grundlehren Math. Wiss.}
\newblock Springer, Cham, 1976.

\bibitem{brezis}
H.~Brezis.
\newblock \textit{ Functional analysis, {Sobolev} spaces and partial differential equations}.
\newblock Universitext. New York, NY: Springer, 2011.

\bibitem{DD}
A.~Debussche and L.~Dettori.
\newblock On the {C}ahn-{H}illiard equation with a logarithmic free energy.
\newblock \textit{ Nonlinear Anal.}, 24(10):1491--1514, 1995.

\bibitem{Demengel}
F.~Demengel and G.~Demengel.
\newblock \textit{ Functional spaces for the theory of elliptic partial differential equations. {Transl}. from the {French} by {Reinie} {Ern{\'e}}}.
\newblock Universitext. Berlin: Springer, 2012.

\bibitem{Denk}
R.~Denk, M.~Hieber, and J.~Pr{\"u}ss.
\newblock \textit{ {$\mathcal{R}$}-boundedness, {Fourier} multipliers and problems of elliptic and parabolic type}, volume 788 of \textit{ Mem. Am. Math. Soc.}
\newblock Providence, RI: American Mathematical Society (AMS), 2003.

\bibitem{Denk2}
R.~Denk, M.~Hieber, and J.~Pr{\"u}ss.
\newblock Optimal {{\(L^{p}\)}}- {{\(L^{q}\)}}-estimates for parabolic boundary value problems with inhomogeneous data.
\newblock \textit{ Math. Z.}, 257(1):193--224, 2007.

\bibitem{Orlicz}
L.~Diening, P.~Harjulehto, P.~H\"ast\"o, and M.~Ru{\v{z}}i{\v{c}}ka.
\newblock \textit{ Lebesgue and {S}obolev spaces with variable exponents}, volume 2017 of \textit{ Lecture Notes in Mathematics}.
\newblock Springer, Heidelberg, 2011.

\bibitem{hyper}
T.~Eiter and R.~Lasarzik.
\newblock Existence of energy-variational solutions to hyperbolic conservation laws.
\newblock \textit{ Calc. Var. Partial Differ. Equ.}, 63(4):40, 2024.
\newblock Id/No 103.

\bibitem{EL18}
E.~Emmrich and R.~Lasarzik.
\newblock Weak-strong uniqueness for the general {Ericksen-Leslie} system in three dimensions.
\newblock \textit{ Discrete and Continuous Dynamical Systems}, 38(9):4617--4635, 2018.

\bibitem{FLR17}
S.~Frigeri, K.~F. Lam, and E.~Rocca.
\newblock On a diffuse interface model for tumour growth with non-local interactions and degenerate mobilities.
\newblock In \textit{ Solvability, regularity, and optimal control of boundary value problems for PDEs. In honour of Prof. Gianni Gilardi}, pages 217--254. Cham: Springer, 2017.

\bibitem{GL17}
H.~Garcke and K.~F. Lam.
\newblock Analysis of a {Cahn-Hilliard} system with non-zero {Dirichlet} conditions modeling tumor growth with chemotaxis.
\newblock \textit{ Discrete and Continuous Dynamical Systems}, 37(8):4277--4308, 2017.

\bibitem{GL17bis}
H.~Garcke and K.~F. Lam.
\newblock Well-posedness of a {C}ahn-{H}illiard system modelling tumour growth with chemotaxis and active transport.
\newblock \textit{ European Journal of Applied Mathematics}, 28(2):284–316, 2017.

\bibitem{GT}
H.~Garcke and D.~Trautwein.
\newblock Numerical analysis for a {C}ahn-{H}illiard system modelling tumour growth with chemotaxis and active transport.
\newblock \textit{ J. Numer. Math.}, 30(4):295--324, 2022.

\bibitem{CahnHilliardOono}
A.~Giorgini, M.~Grasselli, and A.~Miranville.
\newblock The {Cahn}-{Hilliard}-{Oono} equation with singular potential.
\newblock \textit{ Math. Models Methods Appl. Sci.}, 27(13):2485--2510, 2017.

\bibitem{GHW}
A.~Giorgini, J.~He, and H.~Wu.
\newblock Global weak solutions to a {N}avier–{S}tokes–{C}ahn–{H}illiard system with chemotaxis and mass transport: Cross diffusion versus logistic degradation.
\newblock \textit{ Math. Mod. and Meth. Appl. Sci.}, 36(01):57--110, 2026.

\bibitem{Luisa}
D.~H{\"o}mberg, R.~Lasarzik, and L.~Plato.
\newblock On the existence of generalized solutions to a spatio-temporal predator-prey system with prey-taxis.
\newblock \textit{ J. Evol. Equ.}, 23(1):44, 2023.
\newblock Id/No 20.

\bibitem{KS}
E.~Keller and L.~Segel.
\newblock Model for chemotaxis.
\newblock \textit{ J. Theoret. Biol.}, 30:225--234, 1971.

\bibitem{KNP}
N.~Kenmochi, M.~Niezg\'odka, and I.~Paw\l{o}w.
\newblock Subdifferential operator approach to the {C}ahn-{H}illiard equation with constraint.
\newblock \textit{ J. Differential Equations}, 117(2):320--356, 1995.

\bibitem{envar}
R.~Lasarzik.
\newblock On the existence of energy-variational solutions in the context of multidimensional incompressible fluid dynamics.
\newblock \textit{ Math. Methods Appl. Sci.}, 47(6):4319--4344, 2024.

\bibitem{LRR24}
R.~Lasarzik, E.~Rocca, and R.~Rossi.
\newblock Existence and weak-strong uniqueness for damage systems in viscoelasticity.
\newblock \textit{ Nonlinearity}, 38:125016, 2024.

\bibitem{ourpreviousone}
R.~Lasarzik, E.~Rocca, and G.~Schimperna.
\newblock Weak solutions and weak-strong uniqueness for a thermodynamically consistent phase-field model.
\newblock \textit{ Atti Accad. Naz. Lincei, Cl. Sci. Fis. Mat. Nat., IX. Ser., Rend. Lincei, Mat. Appl.}, 33(2):229--269, 2022.

\bibitem{MaMi2009}
A.~Mainik and A.~Mielke.
\newblock Existence results for energetic models for rate-independent systems.
\newblock \textit{ Calc. Var. Partial Differential Equations}, 22(1):73--99, 2005.

\bibitem{MAIMS}
A.~Miranville.
\newblock The {C}ahn–{H}illiard equation and some of its variants.
\newblock \textit{ AIMS Mathematics}, 2(3):479--544, 2017.

\bibitem{MiranvilleM2AN2004}
A.~Miranville and S.~Zelik.
\newblock Robust exponential attractors for {Cahn}-{Hilliard} type equations with singular potentials.
\newblock \textit{ Math. Methods Appl. Sci.}, 27(5):545--582, 2004.

\bibitem{RSS}
E.~Rocca, G.~Schimperna, and A.~Signori.
\newblock On a {C}ahn-{H}illiard-{K}eller-{S}egel model with generalized logistic source describing tumor growth.
\newblock \textit{ J. Differential Equations}, 343:530--578, 2023.

\bibitem{RoubicekBook}
T.~Roub{\'{\i}}{\v{c}}ek.
\newblock \textit{ Nonlinear partial differential equations with applications}, volume 153 of \textit{ ISNM, Int. Ser. Numer. Math.}
\newblock Basel: Birkh{\"a}user, 2nd ed. edition, 2013.

\bibitem{GiulioNew}
G.~Schimperna.
\newblock On a modified {Cahn}-{Hilliard}-{Brinkman} model with chemotaxis and nonlinear sensitivity.
\newblock Preprint, {arXiv}:2411.12505 [math.{AP}] (2024), 2024.

\bibitem{SP}
G.~Schimperna and I.~Paw\l~ow.
\newblock On a class of {C}ahn-{H}illiard models with nonlinear diffusion.
\newblock \textit{ SIAM J. Math. Anal.}, 45(1):31--63, 2013.

\bibitem{SchimpSeg}
G.~Schimperna and A.~Segatti.
\newblock Global attractor for a cahn-hilliard-chemotaxis model with logistic degradation.
\newblock Preprint, {arXiv}:2511.11363 [math.{AP}] (2025), 2025.

\bibitem{SimonComp}
J.~Simon.
\newblock Compact sets in the space {{\({L}^ p(0,{T};{B})\)}}.
\newblock \textit{ Ann. Mat. Pura Appl. (4)}, 146:65--96, 1987.

\end{thebibliography}

\end{document}